\documentclass[letter,11pt]{article}

\usepackage[margin=1in]{geometry}
\usepackage[useregional]{datetime2}
\usepackage[english]{babel}
\usepackage[utf8]{inputenc}
\usepackage{xcolor, graphicx,changepage,xcolor, float, array, amssymb, amsmath, mathrsfs, amsfonts, hyperref, url, amsthm, tikz-cd, fancyhdr, enumerate, bbm, multicol}
\usepackage{tikz-cd}

\makeatletter
\newcommand*{\rom}[1]{\expandafter\@slowromancap\romannumeral #1@}
\makeatother

\makeatletter
\renewcommand*\env@matrix[1][*\c@MaxMatrixCols c]{%
	\hskip -\arraycolsep
	\let\@ifnextchar\new@ifnextchar
	\array{#1}}
\makeatother

\theoremstyle{definition}
\newtheorem{theorem}{Theorem}[section]
\newtheorem{proposition}{Proposition}[section]
\newtheorem{lemma}[theorem]{Lemma}
\newtheorem{definition}[theorem]{Definition}
\newtheorem{conjecture}[theorem]{Conjecture}
\newtheorem{question}[theorem]{Question}
\newtheorem{corollary}[theorem]{Corollary}

\theoremstyle{remark}
\newtheorem{remark}[theorem]{Remark}

\usepackage{setspace}
\newcommand{\N}{\mathbb{N}}
\newcommand{\Z}{\mathbb{Z}}
\newcommand{\Q}{\mathbb{Q}}

\title{Undecidability in the Ramsey theory of polynomial equations and Hilbert's tenth problem}
\author{Sohail Farhangi, Steve Jackson, Bill Mance}

\begin{document}

\maketitle
\begin{abstract}
    We show that several sets of interest arising from the study of partition regularity and density Ramsey theory of polynomial equations over integral domains are undecidable. In particular, we show that the set of homogeneous polynomials $p \in \mathbb{Z}[x_1,\cdots,x_n]$ for which the equation $p(x_1,\cdots,x_n) = 0$ is partition regular over $\mathbb{Z}\setminus\{0\}$ is undecidable conditional on Hilbert's tenth problem for $\mathbb{Q}$. For other integral domains, we get the analogous result unconditionally. More generally, we determine the exact lightface complexity of the various sets of interest. For example, we show that the set of homogeneous polynomials $p \in \mathbb{F}_q(t)[x_1,\cdots,x_n]$ for which the equation $p(x_1,\cdots,x_n) = 0$ is partition regular over $\mathbb{F}_q(t)\setminus\{0\}$ is $\Pi_2^0$-complete.

    We also prove a compactness principle and a uniformity principle for density Ramsey theory on countable cancellative left amenable semigroups, which is of independent interest.  
\end{abstract}
\section{Introduction}
An active area of research in Ramsey theory is the partition regularity of polynomial equations. A system of polynomial equations is \textbf{partition regular} over an integral domain $R$ if for any finite partition of the form $R\setminus\{0\} = \bigcup_{i = 1}^\ell C_i$, there exists some $1 \le i_0 \le \ell$ that contains a solution to the given system of equations.\footnote{We work with $R\setminus\{0\}$ instead of $R$ since most polynomials $p \in R[x_1,\cdots,x_n]$ that we consider will satisfy $p(0,\cdots,0) = 0$, and we want to avoid trivialities.} One of the earliest results about partition regular equations over $\mathbb{Z}$ was Schur's Theorem \cite{SchurThm}, which states that the equation $x+y = z$ is partition regular over $\mathbb{Z}\setminus\{0\}$. 
Later, van der Waerden \cite{vdWThm} showed a similar result for arithmetic progressions. In particular, van der Waerden showed that for $m \in \mathbb{N}$ and any finite partition of the form $\mathbb{Z}\setminus\{0\} = \bigcup_{i = 1}^\ell C_i$, there exists a $1 \le i_0 \le \ell$ and $a,d \in \mathbb{N}$ such that $\{a+jd\}_{j = 0}^m \subseteq C_{i_0}$. We observe that for $m \ge 3$, length $m$ arithmetic progressions can also be characterized as roots of the polynomial

\begin{equation}
    p_m(x_1,\cdots,x_m) = \sum_{j = 1}^{m-2}\left(x_{j+2}-2x_{j+1}+x_j\right)^2.
\end{equation}
Non-degenerate length $m$ arithmetic progressions correspond to a root $z_1,\cdots,z_m$ of $p_m$ for which $z_i \neq z_j$ when $i \neq j$, so van der Waerden's Theorem can also be viewed as a statement about the (injective) partition regularity of a particular polynomial equation over $\mathbb{Z}\setminus\{0\}$.

Later, Brauer \cite{Brauer'sThm} proved a common refinement of Schur's Theorem and van der Waerden's Theorem, and Rado \cite{RadosTheorem} classified which finite systems of linear equations are partition regular over $\mathbb{N}$. 
While we know precisely which finite linear systems of equations are partition regular, we are still far from a complete understanding of which polynomial equations are partition regular. 
For example, it is a classical open problem of Erd\H{o}s and Graham \cite{OldAndNewRamsey} to determine whether or not the Pythagorean equation $x^2+y^2 = z^2$ is partition regular over $\mathbb{Z}\setminus\{0\}$, Hindman's conjecture can be restated as the question of whether or not the equation $(z-x-y)^2+(w-xy)^2 = 0$ is partition regular over $\mathbb{N}$, it is a question of Farhangi and Magner \cite{ax+by=cwmzn} to determine whether or not $16x+17y = wz^8$ is partition regular over $\mathbb{Z}\setminus\{0\}$, and it is a conjecture of Prendiville \cite{prendiville2020counting} that for any $a,b,c \in \mathbb{Z}\setminus\{0\}$, the equation $a(x^2-y^2) = bz^2+cw$ is partition regular over $\mathbb{N}$ (and hence also over $\mathbb{Z}\setminus\{0\}$). 
Furthermore, Chow, Lindqvist, and Prendiville \cite{RamseyWaring} showed that $x_1^2+x_2^2+x_3^2+x_4^2 = x_5^2$ is partition regular over $\mathbb{N}$, but it is still an open problem to determine whether or not $x_1^2+x_2^2+x_3^2 = x_4^2$ is partition regular over $\mathbb{N}$.
We refer the reader to \cite{frantzikinakis2023partition,frantzikinakis2024partition} and the references therein for a discussion of the partial progress made regarding the partition regularity of the Pythagorean equation, to \cite{NonlinearRado,NonStandardRamseyTheory,RadoFunctionalsAndApplications} with regards to general progress on determining which polynomial equations are partition regular, and to \cite{MoreiraAnnals,bowen2022monochromatic,AlweissSumsAndProducts,MonochromaticSumsAndProductsInTheRationals} with regards to progress on Hindman's Conjecture.

It is natural to ask how difficult it is to characterize which systems of polynomial equations are partition regular. 
The lightface hierarchy (see Section 2.6) provides a concrete framework for discussing the complexity of various sets of interest. 
Hilbert's 10th problem was solved by determining the exact placement in the lightface hierarchy of the set of polynomials with coefficients in $\mathbb{Z}$ that also have a root in $\mathbb{Z}$, and this set was shown to be $\Sigma_1^0$-complete, from which undecidability follows. One of the goals of this paper to reduce Hilbert's 10th problem over $\mathbb{Q}$ to a question about the partition regularity of certain homogeneous polynomial equations over $\mathbb{Z}\setminus\{0\}$, from which we will deduce that it is conditionally\footnote{While Hilbert's 10th problem over $\mathbb{Z}$ is known to be undecidable, it is still an open problem to determine whether or not Hilbert's 10th problem over $\mathbb{Q}$ is undecidable.} undecidable to determine whether or not a given polynomial equation is partition regular over $\mathbb{Z}$.
In fact, we determine the exact placement in the lightface hierarchy of the set of partition regular polynomials over $\mathbb{Z}$,\footnote{The ``set of partition regular polynomials over $\mathbb{Z}$'' is a convenient way of referring to the set of polynomials $p$ for which the equation $p(x_1,\cdots,x_n) = 0$ is partition regular over $\mathbb{Z}$.} and show that it is $\Pi_2^0$-complete.
We will also study the question of partition regularity of polynomial equations over a general countable integral domain $R$, and in many cases we will unconditionally show that it is undecidable to determine whether or not a given polynomial equation is partition regular over $R$. 
We also obtain similar undecidability results when we consider $\ell$-partition regularity of polynomial equations, i.e., when we only require a root to be found in any partition containing at most $\ell$ cells.
In fact, we determine the exact placement in the lightface hierarchy of the set of $\ell$-partition regular polynomials by showing that it is $\Sigma_1^0$-complete for every $\ell \in \mathbb{N}$. Similarly, we show that the set of partition regular polynomials over $R$ in these cases is $\Sigma_1^0$-hard. Furthermore, for many (but not all) of these cases we show that the set of partition regular polynomials over $R$ is $\Pi_2^0$-complete.

Another active area of research in Ramsey theory is density Ramsey theory. For a set $A \subseteq \mathbb{N}$, the (natural) upper density of $A$ is given by $\overline{d}(A) := \limsup_{N\rightarrow\infty}\frac{|A\cap[1,N]|}{N}$. Density Ramsey theory, roughly speaking, is the study of what structures can be found in sets $A \subseteq \mathbb{N}$ satisfying $\overline{d}(A) > 0$. It is worth noting that the function $\overline{d}$ is subadditive, so for finite partitions of the form $\mathbb{N} = \bigcup_{i = 1}^\ell C_i$, there exists some $1 \le i_0 \le \ell$ for which $\overline{d}(C_{i_0}) > 0$. Consequently, a positive result in density Ramsey theory also yields a positive result in partition regular Ramsey theory, but the converse is not true in general. Roth \cite{Roth2} showed that if $A \subseteq \mathbb{N}$ has positive upper density, then $A$ contains a 3-term arithmetic progression. Szemer\'edi \cite{Szemeredi4Term,SzemerediThm} showed that if $A \subseteq \mathbb{N}$ has positive upper density, then $A$ contains arbitrarily long arithmetic progressions.
Bergelson and Leibman \cite{PolynomialSzemeredi} showed that if $A \subseteq \mathbb{N}$ has positive upper density, then for any collection of  polynomials $F := \{p_i(x)\}_{i = 1}^m \subseteq x\mathbb{Z}[x]$ there exists $a,d \in \mathbb{N}$ such that $\{a+p_i(d)\}_{i = 1}^m \subseteq A$. We observe that configurations of the form $\{d\}\cup\{a+p_i(d)\}_{i = 1}^m$ can be identified with roots of the polynomial

\begin{equation}
    P_F(x_1,\cdots,x_{m+1}) = \sum_{i = 1}^m(x_i-p_i(x_1)-x_2)^2.
\end{equation}
Furthermore, the non-degeneracy of such configurations corresponds to a root $z_1,\cdots,z_{m+1}$ of $P_F$ with $z_i \neq z_j$ when $i \neq j$. Consequently, the result of Bergelson and Leibman is equivalent to showing that for any $A \subseteq \mathbb{N}$ with $\overline{d}(A) > 0$, there exists $z_2,\cdots,z_{m+1} \in A$ and $z_1 \in \mathbb{N}$ for which $P_F(z_1,\cdots,z_{m+1}) = 0$.
For more recent results in density Ramsey theory, we refer the reader to \cite[Corollary 1.6]{TsinasHardyFieldSequences} and \cite[Corollary 2.12]{frantzikinakis2022joint} and the references therein.

It is natural to ask about the placement in the lightface hierarchy of the set $IADR$ of polynomials $P \in \mathbb{Z}[x_1,\cdots,z_n]$ possessing an injective root in any $A \subseteq \mathbb{N}$ with $\overline{d}(A) > 0$, as well as the set $IADR_1$ of polynomials $P$ for which every $A \subseteq \mathbb{N}$ with $\overline{d}(A) > 0$ contains $z_2,\cdots,z_n \in A$ for which there exists $z_1 \in \mathbb{N}$ for which $P(z_1,\cdots,z_n) = 0$. 
We reduce a variant of Hilbert's Tenth problem over $\mathbb{Q}$ to the question of whether or not a given polynomial $P$ is in $IADR$, as well as the question of whether or not $P$ is in $IADR_1$, thereby showing that $IADR$ and $IADR_1$ are $\Pi_2^0$-complete, hence undecidable.
We also consider similar sets over a countable integral domains $R$, and show that they are also $\Pi_2^0$-complete. 
We also obtain similar undecidability results for a multiplicative notion of density (rather than the additive notion we considered here), and for the situation in which we require a root in all sets $A \subseteq \mathbb{N}$ with $\overline{d}(A) \ge \delta$ for some fixed $\delta \in (0,1]$.

In order to prove our results regarding density regularity, we needed a general compactness principles for density Ramsey Theory on countable cancellative abelian semigroups. Since compactness principles and uniformity principles are of independent interest in Ramsey theory, we proved such principles for an arbitrary countable cancellative left amenable semigroup.

The structure of the paper is as follows. 
In Section 2 we collect well known facts from the literature about amenable semigroups, embedding semigroups in groups, measure preserving systems, Ramsey Theory, Descriptive Set Theory, and Hilbert's tenth problem. 
In Section 3 we prove a compactness principle for density Ramsey theory as Theorem \ref{DensityCompactnessPrinciple}, a uniformity principle as Theorem \ref{Uniformity}, and extend a classical list of equivalent notions of density regularity for $\mathbb{Z}$ to the setting of countable cancellative left amenable semigroups as Theorem \ref{7Equivalences}. 
In Section 4 we prove our main results regarding the descriptive complexity of various sets of interest regarding the Ramsey theory of polynomial equations. 
In Section 5 we discuss the implications of the results of Section 4 and pose several questions.

\textbf{Acknowledgements:} The authors would like to thank many people for helpful discussions throughout the writing of this paper. We would like to thank Ethan Ackelsberg, Vitaly Bergelson, Daniel Glasscock, John Johnson, Bhawesh Mishra, Joel Moreira, Brian Tyrrell, and Alexandra Shlapentokh for help finding and understanding various references in the literature. We would also like to thank Dilip Raghavan for asking Question \ref{QuestionAboutStrengtheningResults}(i). Furthermore, we would like to thank Jonathan Chapman for a fruitful conversation that led to Conjecture \ref{ConjectureForPartitionRegularity}. Lastly, we would like to thank Philip Dittman for reading a draft of this work and providing helpful feedback.

After the first version of this work was announced, it was brought to our attention that the topic of natural extensions of continuous/measure preserving actions of semigroups that embed in a group has already been thoroughly investigated by Brice\~no, Bustos-Gajardo, and Donoso-Echenique \cite{briceno2025extensibility,briceno2025natural,echenique2024natural}. We have made some minor changes in our presentation of this topic in Subsection \ref{MPSSubsection} in order to be consistent with this previous work.

SF and BM acknowledge being supported by the grant
2019/34/E/ST1/00082 for the project “Set theoretic methods in dynamics and number theory,” NCN (The
National Science Centre of Poland). SJ acknolwedges being support by NSF grant: DMS-1800323.
\section{Preliminaries}
\subsection{Amenable Semigroups and tilings of amenable groups}
Let $(S,\cdot)$ be a semigroup. For $A \subseteq S$ and $s \in S$, we write $As = \{as\ |\ a \in A\}$, $sA = \{sa\ |\ a \in A\}$, $s^{-1}A = \{a \in S\ |\ sa \in A\}$, and $As^{-1} = \{a \in S\ |\ as \in A\}$. We let $\mathscr{P}_f(S)$ denote the collection of finite subsets of $S$. Now let us assume that $S$ is countably infinite and cancellative. Given $K,F \in \mathscr{P}_f(S)$ and an $\epsilon > 0$, the set $F$ is \textbf{$(K,\epsilon)$-invariant} if for any $k \in K$ we have $|kF\triangle F| < \epsilon|F|$. A \textbf{left F\o lner sequence} is a collection $\mathcal{F} = (F_n)_{n = 1}^\infty \subseteq \mathscr{P}_f(S)$ satisfying

\begin{equation}
    \lim_{n\rightarrow\infty}\frac{|sF_n\triangle F_n|}{|F_n|} = 0,
\end{equation}
for all $s \in S$. Equivalently, $(F_n)_{n = 1}^\infty$ is a left F\o lner sequence if for any $K \in \mathscr{P}_f(S)$ and any $\epsilon > 0$, there exists $N \in \mathbb{N}$ such that $F_n$ is $(K,\epsilon)$-invariant for all $n \ge N$. The cancellative semigroup $S$ is \textbf{left amenable} if it admits a left F\o lner sequence. It is well known that all abelian semigroups are (left )amenable. Given a left F\o lner sequence $\mathcal{F}$ and a set $A \subseteq S$, the \textbf{$\mathcal{F}$-upper density} of $A$ is given by

\begin{equation}\label{EquationDefiningUpperFolnerDensity}
    \overline{d}_{\mathcal{F}}(A) := \limsup_{n\rightarrow\infty}\frac{|A\cap F_n|}{|F_n|}.
\end{equation}
If the limit in Equation \eqref{EquationDefiningUpperFolnerDensity} exists, then we denote its value by $d_{\mathcal{F}}(A)$, which is known as the \textbf{$\mathcal{F}$-density} of $A$. The \textbf{upper Banach density} of $A$ is given by

\begin{equation}
    d^*(A) := \sup\left\{\overline{d}_{\mathcal{F}}(A)\ |\ \mathcal{F}\text{ is a left F\o lner sequence}\right\}.
\end{equation}
When we work with an integral domain $R$, we use $d^*$ to denote the upper Banach density in the group $(R,+)$, and we use $d^*_\times$ to denote the upper Banach density in the semigroup $(R\setminus\{0\},\cdot)$. We also record here the following alternative description of upper Banach density that first appeared in \cite{GriesmerRecurrenceRigidityPopularDifferences} in the case of $G = \mathbb{Z}$, and appeared in our desired generality as \cite[Theorem G]{johnson2016revisiting} (see also Theorem 3.5 and Corollary 3.6 of \cite{InterplayBetweenAdditiveAndMultiplicativeLargeness}).

\begin{lemma}\label{AlternativeCharacterizationOfUBD}
    If $S$ is a countable cancellative left amenable semigroup, and $A \subseteq S$, then

    \begin{equation}
        d^*(A) = \sup\left\{\alpha \ge 0\ |\ \forall F\in\mathscr{P}_f(S), \exists s \in S\text{ such that }|Fs\cap A| \ge \alpha|F|\right\}.
    \end{equation}
\end{lemma}

We observe that if $(F_n)_{n = 1}^\infty$ is a left F\o lner sequence in a group $G$, then $(F_n^{-1})_{n = 1}^\infty$ is a right F\o lner sequence. Consequently, a group is left amemable if and only if it is right amenable, so we only speak about groups being amenable or non-amenable. We also mention that amenability can be discussed for non-cancellative semigroups, and we refer the reader to \cite[Section 4.22]{AmenabilityByPaterson} for a discussion of some of the subtleties that arise in this situation.

\begin{definition}
    A \textit{tiling} $\mathcal{T}$ of a group $G$ is determined by two objects:
    \begin{enumerate}[(1)]
        \item a finite collection $\mathcal{S}(\mathcal{T})$ of finite subsets of $G$ containing the identity $e$, called \textit{the shapes},

        \item a finite collection $\mathcal{C}(\mathcal{T}) = \{C(S)\ |\ S \in \mathcal{S}(\mathcal{T})\}$ of disjoint subsets of $G$, called \textit{center sets} (for the shapes).
    \end{enumerate}
    The tiling $\mathcal{T}$ is then the family $\{(S,c)\ |\ S \in \mathcal{S}(\mathcal{T})\ \&\ c \in C(S)\}$ provided that $\{Sc\ |\ (S,c) \in \mathcal{T}\}$ is a partition of $G$. A \textit{tile} of $\mathcal{T}$ refers to a set of the form $T = Sc$ with $(S,c) \in \mathcal{T}$, and in this case we may also write $T \in \mathcal{T}$. A sequence $(\mathcal{T}_k)_{k = 1}^\infty$ of tilings is \textit{congruent} if each tile of $\mathcal{T}_{k+1}$ is a union of tiles of $\mathcal{T}_k$, and in this case we further assume without loss of generality that $\bigcup_{S \in \mathcal{S}(\mathcal{T}_{k+1})}C(S) \subseteq \bigcup_{S \in \mathcal{S}(\mathcal{T}_k)}C(S)$. Given a tiling $\mathcal{T}$, a finite set $F \subseteq G$, and a $\epsilon > 0$, we say that $\mathcal{T}$ \textbf{$\epsilon$-tiles} $F$ if there is are sets $F'$ and $F''$ that are each a union of tiles of $\mathcal{T}$ for which $F' \subseteq F \subseteq F''$, $|F\setminus F'| < \epsilon|F|$, and $|F''\setminus F| < \epsilon|F|$.
\end{definition}

We see that any group $G$ has a trivial tiling $\mathcal{T}$ in which $\mathcal{S}(\mathcal{T}) = \{\{e\}\}$ and $\mathcal{C}(\mathcal{T}) = \{G\}$. When the group $G$ is amenable, we look for more interesting tilings by requiring that the shapes of the tiling be $(K,\epsilon)$-invariant for some $K \in \mathscr{P}_f(G)$ and $\epsilon > 0$. We now recall a special case of a result of Downarowicz, Huczek, and Zhang about such tilings.

\begin{theorem}[{\cite[Theorem 5.2]{TilingAmenableGroups}}]\label{CongruentTilingTheorem}
    Let $G$ be a countably infinite amenable group. Fix a converging to zero sequence $\epsilon_k > 0$ and a sequence $K_k$ of finite subsets of $G$. There exists a congruent sequence of tilings $(\mathcal{T}_k)_{k = 1}^\infty$ of $G$ such that the shapes of $\mathcal{T}_k$ are $(K_k,\epsilon_k)$-invariant.
\end{theorem}

\begin{lemma}\label{GoodTilingLemma}
    Let $G$ be an amenable group, let $\mathcal{T}$ be a tiling of $G$, and let $T = \bigcup_{T' \in \mathcal{S}(\mathcal{T})}T'$. If $F \in \mathscr{P}_f(G)$ is $\left(TT^{-1},\epsilon|TT^{-1}|^{-1}\right)$-invariant, then $\mathcal{T}$ $\epsilon$-tiles $F$.
\end{lemma}

\begin{proof}
    Let $F_0 = \bigcap_{g \in TT^{-1}}gF = \bigcap_{g \in TT^{-1}}g^{-1}F$, and recall that $|F\setminus F_0| < \epsilon|F|$. For $h \in F_0$, let $T_h \in \mathcal{S}(\mathcal{T})$ and $c_h \in C(T_h)$ be such that $h \in T_hc_h$, hence $c_h \in T_h^{-1}h$ and $T_hc \subseteq T_hT_h^{-1}h \subseteq TT^{-1}h \subseteq F$. It follows that if $F'$ is the union of all tiles that intersect $F_0$ in at least one point, then $F_0 \subseteq F' \subseteq F$, and $|F\setminus F'| \le |F\setminus F_0| < \epsilon|F|$. Similarly, we see that if $f \in F$ and $T_f \in \mathcal{S}(\mathcal{T})$ and $c_f \in C(T_f)$ are such that $f \in T_fc_f$, then $T_fc_f \subseteq T_fT_f^{-1}f \subseteq TT^{-1}F$. It follows that if $F''$ is the union of all tiles that intersect $F$ in at least one point, then $F \subseteq F'' \subseteq TT^{-1}F$ and $|F''\setminus F| \le |TT^{-1}F\setminus F| = |TT^{-1}F\triangle F| < \epsilon|F|$.
\end{proof}

A tiling $\mathcal{T}$ is a \textbf{monotiling} if $\mathcal{S}(\mathcal{T}) = \{T\}$. In this case $T$ is called a \textbf{monotile}, and $C := C(T)$ is a \textbf{center set for $T$}. An amenable group $G$ is \textbf{congruently monotileable} if Theorem \ref{CongruentTilingTheorem} can be satisfied with $(\mathcal{T}_k)_{k = 1}^\infty$ being a sequence of monotilings. It is shown in \cite{VirtuallyNilpotentIsCongruentlyMonotilable} that all countable virtually nilpotent groups are congruently monotileable. Later on, we will only need to use the fact that all countable abelian groups are congruently monotileable. We will also be using the fact that if $G$ is a monotileable amenable group and $F \in \mathscr{P}_f(G)$, then there exists a monotile $T$ of $G$ with $F \subseteq T$.
\subsection{Embedding semigroups in groups}
The main purpose of this section is to show that a cancellative left amenable semigroup $S$ embeds nicely into an amenable group $G$. This will allow us to prove results for $S$, by first proving them for $G$, where we can make use of results such as Theorems \ref{CongruentTilingTheorem} and \ref{PointwiseErgodicTheorem}

A semigroup $S$ is \textbf{left reversible} if every pair of principal right ideals of $S$ intersect. In other words, $S$ is left reversible if and only if for any $s,t \in S$ we have $sS\cap tS \neq \emptyset$. A group $G$ is a \textbf{group of right quotients of $S$} if $S \subseteq G$ and every $g \in G$ has the form $g = st^{-1}$ for some $s,t \in S$.  It is a classical result of Ore (cf. \cite[Theorem 1.23]{AlgebraicTheoryOfSemigroupsVol1}) that if $S$ is a cancellative left reversible semigroup, then $S$ embeds in a group of its right quotients. It is worth noting that left reversibility is not a necessary condition for a cancellative semigroup to embed in a group, but Dubreil (cf. \cite[Theorem 1.24]{AlgebraicTheoryOfSemigroupsVol1}) showed that a cancellative semigroup $S$ embeds in a group of its right quotients only if it is left reversible. Furthermore, \cite[Theorem 1.25]{AlgebraicTheoryOfSemigroupsVol1} tells us that if $G_1$ and $G_2$ are groups of right quotients of a left reversible cancellative semigroup $S$, then $G_1$ and $G_2$ are isomorphic as groups. Consequently, we refer to the group of right quotients of $S$ when $S$ is cancellative and left reversible. It is well known (cf. \cite[Proposition 1.23]{AmenabilityByPaterson}) that every left amenable semigroup is left reversible. Furthermore, if $\mathcal{F}$ is a left F\o lner sequence in the cancellative left amenable semigroup $S$, then $\mathcal{F}$ is also a F\o lner sequence in $G$, the group of right quotients of $S$. To see this, we observe that if $F \subseteq S$ is $(K,\epsilon)$-invariant for some $K \in \mathscr{P}_f(S)$ and $\epsilon > 0$, then for any $k_1,k_2 \in K$, we have 

     \begin{equation}
         |k_1k_2^{-1}F\triangle F| = |k_2^{-1}F\triangle k_1^{-1}F| \le |k_2^{-1}F\triangle F|+|F\triangle k_1^{-1}F| = |F\triangle k_2F|+|k_1F\triangle F| < 2\epsilon|F|.
     \end{equation}

If $S$ is a semigroup, then $A \subseteq S$ is \textbf{(left) thick} if for any $F \in \mathscr{P}_f(S)$ there exists $s \in S$ for which $Fs \subseteq A$. It is easily checked that $A \subseteq S$ is thick if and only if for any $F \in \mathscr{P}_f(S)$, there exists $s \in A$ for which $Fs \subseteq A$. If $S$ is cancellative and left reversible, and $G$ is the group of right quotients of $S$, then $S$ is a thick subset of $G$. We summarize the preceding discussion with the following result.

\begin{theorem}\label{EmbeddingSemigroupsIntoGroups}
    Let $S$ be a cancellative semigroup. If $S$ is left reversible, then it embeds in its group of right quotients $G$ as a thick subset. If $S$ is left amenable, then it is also left reversible, and its group of right quotients $G$ is an amenable group.
\end{theorem}

While we will not directly make use of the following result, we record it here since we could not find a reference in the literature. This result further illustrates how nicely a semigroup embeds in its group of right quotients.

\begin{lemma}\label{ExtendingHomomorphismsLemma}
    Let $S$ be a cancellative left reversible semigroup and let $G$ be the group of right quotients of $S$. If $h:S\rightarrow S$ is a homomorphism, then the map $h:G\rightarrow G$ given by $h\left(st^{-1}\right) = h(s)h(t)^{-1}$ is also a homomorphism.
\end{lemma}

\begin{proof}
    First we will show that $h$ is a well defined map. Let $s,t,x,y \in S$ be such that $st^{-1} = xy^{-1}$. We want to show that $h\left(st^{-1}\right) = h\left(xy^{-1}\right)$, which is equivalent to showing $h(s)h(t)^{-1} = h(x)h(y)^{-1}$, so it suffices to show that $h\left(xy^{-1}t\right) = h(x)h(y)^{-1}h(t)$. Since $S$ is left reversible, let $r_1,r_2 \in S$ be such that $tr_1 = yr_2$, and observe that $xy^{-1}tr_1 = xr_2$. We see that

    \begin{alignat*}{2}
        &h(x)h(y)^{-1}h(t) = h(x)h(y)^{-1}h(t)h(r_1)h(r_1)^{-1} = h(x)h(y)^{-1}h(tr_1)h(r_1)^{-1}\\
        =& h(x)h(y)^{-1}h(yr_2)h(r_1)^{-1} = h(x)h(y)^{-1}h(y)h(r_2)h(r_1)^{-1} = h(x)h(r_2)h(r_1)^{-1}\\
        =&h(xr_2)h(r_1)^{-1} = h\left(xy^{-1}tr_1\right)h(r_1)^{-1} = h\left(xy^{-1}t\right)h(r_1)h(r_1)^{-1} = h\left(xy^{-1}t\right).
    \end{alignat*}

    Now that we have verified that $h:G\rightarrow G$ is well defined, it remains to check that it is a homomorphism. Let $s \in S$ be arbitrary, and observe that
    
    \begin{equation}
        h(e) = h\left(ss^{-1}\right) = h(s)h(s)^{-1} = e.
    \end{equation}
    We also see that for any $s,t \in S$ we have
    
    \begin{equation}
        h\left(s^{-1}\right) = h\left(t(st)^{-1}\right) = h(t)h(st)^{-1} = h(t)(h(s)h(t))^{-1} = h(t)h(t)^{-1}h(s)^{-1} = h(s)^{-1}.
    \end{equation}
    Now let $a,b,x,y \in S$ be arbitrary, and we will show that $h\left(ab^{-1}xy^{-1}\right) = h\left(ab^{-1}\right)h\left(xy^{-1}\right)$. Let $f,g \in S$ be such that $ab^{-1}xy^{-1} = fg^{-1}$. Since $S$ is left reversible, let $r_1,r_2 \in S$ be such that $fr_1 = ar_2$, and observe that $r_2 = a^{-1}fr_1 = b^{-1}xy^{-1}gr_1$. We see that

    \begin{alignat*}{2}
        &h\left(fg^{-1}\right) = h(f)h(g)^{-1}h(g)h(r_1)h(r_1)^{-1}h(g)^{-1} = h(fr_1)h(r_1)^{-1}h(g)^{-1} = h(ar_2)h(r_1)^{-1}h(g)^{-1}\\
        =&h(a)h(r_2)h(gr_1)^{-1} = h(a)h\left(r_2(gr_1)^{-1}\right) = h(a)h\left(b^{-1}xy^{-1}\right).
    \end{alignat*}
    In particular, we have shown that $h\left(ab^{-1}xy^{-1}\right) = h(a)h\left(b^{-1}xy^{-1}\right)$. We now see that

    \begin{alignat*}{2}
        &h\left(b^{-1}xy^{-1}\right) = h\left(eb^{-1}xy^{-1}\right) = h\left(yx^{-1}be\right)^{-1} = \left(h(y)h\left(x^{-1}be\right)\right)^{-1} = \left(h(y)h\left(x^{-1}b\right)\right)^{-1}\\
        =&h\left(x^{-1}b\right)^{-1}h(y)^{-1} = h(b^{-1}x)h(y)^{-1}.
    \end{alignat*}
    Up to now, we have shown that $h\left(ab^{-1}xy^{-1}\right) = h(a)h\left(b^{-1}x\right)h(y)^{-1}$. Since $h\left(ab^{-1}\right)h\left(xy^{-1}\right) = h(a)h(b)^{-1}h(x)h(y)^{-1}$, it only remains to show that $h\left(b^{-1}x\right) = h(b)^{-1}h(x)$. To this end, let $s \in S$ be such that $b^{-1}xs \in S$ and observe that

    \begin{alignat*}{2}
        &h\left(b^{-1}x\right) = h\left(b^{-1}xss^{-1}\right) = h\left(b^{-1}xs\right)h(s)^{-1}\text{, so}\\
        &h(b)h\left(b^{-1}x\right)h(s) = h(b)h\left(b^{-1}xs\right) = h\left(bb^{-1}xs\right) = h(xs) = h(x)h(s)\text{, hence}\\
        &h\left(b^{-1}x\right) = h(b)^{-1}h(x).
    \end{alignat*}
\end{proof}
\subsection{Measure preserving systems and their natural extensions}\label{MPSSubsection}
Let $(S,\cdot)$ be a countably infinite semigroup and let $(X,\mathscr{B},\mu)$ be a standard probability space. A measure preserving action of $S$ on $X$ is a collection of measurable maps $(T_s)_{s \in S}$ satisfying the following properties:
\begin{enumerate}[(i)]
    \item For all $s \in S$ and all $A \in \mathscr{B}$, $\mu(A) = \mu(T_s^{-1}A)$.

    \item For all $s,t \in S$, $T_sT_t = T_{st}$.

    \item If $S$ has an identity element $e$, then $T_e$ is the identity map.
\end{enumerate}
A \textbf{(measure preserving) $S$-system} is a tuple $(X,\mathscr{B},\mu,(T_s)_{s \in S})$, where $(X,\mathscr{B},\mu)$ is a standard probability space, and $(T_s)_{s \in S}$ is a measure preserving action of $S$ on $X$. The $S$-system $(X,\mathscr{B},\mu,(T_s)_{s \in S})$ is \textbf{ergodic} if the $A \in \mathscr{B}$ for which we have $\mu(T_s^{-1}A\triangle A) = 0$ for all $s \in S$ satisfy $\mu(A) \in \{0,1\}$. Let us now recall a special case of the pointwise ergodic theorem of Lindenstrauss \cite{AmenableBirkhoff}.

\begin{theorem}[Pointwise ergodic theorem]\label{PointwiseErgodicTheorem}
    Let $G$ be a countable amenable group, let $(X,\mathscr{B},\allowbreak\mu,(T_g)_{g \in G})$ be an ergodic $G$-system. Let $\mathcal{F} = (F_n)_{n = 1}^\infty$ be a tempered left F\o lner sequence, i.e., $\mathcal{F}$ satisfies
    
    \begin{equation}
        \sup_{n \in \mathbb{N}}\frac{|\bigcup_{k = 1}^{n-1}F_k^{-1}F_n|}{|F_n|} < \infty.
    \end{equation}
    If $f \in L^1(X,\mu)$, then for a.e. $x \in X$ we have

    \begin{equation}
        \lim_{N\rightarrow\infty}\frac{1}{|F_N|}\sum_{g \in F_N}f(T_gx) = \int_Xfd\mu.
    \end{equation}
\end{theorem}

Being able to use the pointwise ergodic theorem is one of the reasons that we like to work with ergodic $G$-systems. While not every $G$-system is ergodic, it is well known that any $G$-system can be viewed as a direct integral of ergodic $G$-systems.

\begin{theorem}[{Ergodic decomposition \cite{GeneralErgodicDecomposition}}]\label{ErgodicDecomposition}
    Let $G$ be a countable group and let $(X,\mathscr{B},\mu,(T_g)_{g \in G})$ be a $G$-system. There exists a standard probability space $(Y,\mathscr{A},\nu)$ and a collection of pairwise mutually singular probability measures $\{\mu_y\}_{y \in Y}$ on $(X,\mathscr{B})$ such that for any $A \in \mathscr{B}$ we have

    \begin{equation}
        \mu(A) = \int_Y\mu_y(A)d\nu(y),
    \end{equation}
    and for each $y \in Y$, we have that $(X,\mathscr{B},\mu_y,(T_g)_{g \in G})$ is an ergodic $G$-system.
\end{theorem}

If $\mathcal{X} := (X,\mathscr{B},\mu,(T_s)_{s \in S})$ and $\mathcal{Y} := (X,\mathscr{A},\nu,(R_s)_{s \in S})$ are $S$-systems, then $\mathcal{X}$ is a \textbf{factor} of $\mathcal{Y}$ if there exists a measurable map $\pi:Y\rightarrow X$ for which $\nu(\pi^{-1}A) = \mu(A)$ for all $A \in \mathscr{B}$, and $\pi R_s = T_s\pi$ for all $s \in S$. 
If $\mathcal{X}$ is a factor of $\mathcal{Y}$, then $\mathcal{Y}$ is an \textbf{extension} of $\mathcal{X}$. 
Now let $G$ be a countably infinite group and let $S \subseteq G$ be a subsemigroup that generates $G$ as a group. If $\mathcal{X} = (X,\mathscr{B},\mu,(T_s)_{s \in S})$ is an $S$-system and $\mathcal{Y} = (Y,\mathscr{A},\nu,(R_g)_{g \in G})$ is a $G$-system for which $\mathcal{Y}' := (Y,\mathscr{A},\nu,(R_g)_{g \in S})$ is an extension of $\mathcal{X}$, then $\mathcal{Y}$ is an \textbf{invertible extension} of $\mathcal{X}$. 
The \textbf{natural extension} of $\mathcal{X}$ is an invertible extension $\mathcal{Y}$ with factor map $\pi:Y\rightarrow X$ such that the following holds:
If $\mathcal{Z}$ is an invertible extension of $\mathcal{X}$ with factor map $\pi_1:Z\rightarrow X$, then there exists a unique factor map $\pi_2:Z\rightarrow Y$ of $\mathcal{Z}$ to $\mathcal{Y}$ such that $\pi_1 = \pi\circ\pi_2$.

It is a classical result of Rohlin \cite{rohlin1961exact} (see also \cite[Chapter 10.4]{CornfeldFominSinai}) that if $(S,\cdot) = (\mathbb{N},+)$ then any $S$-system $\mathcal{X}$ has a natural extension $\mathcal{Y}$. 
When $(S,\cdot) = (\mathbb{N}^d,+)$, Lacroix \cite{lacroix1995natural} showed that any $S$-system $\mathcal{X}$ has a natural extension $\mathcal{Y}$. 
Recently, Brice\~no, Bustos-Gajardo, and Donoso-Echenique \cite{briceno2025extensibility,briceno2025natural,echenique2024natural} gave a complete treatment of the topic of natural extensions of countable semigroup actions in the topological and measurable categories.
Here we give a self contained construction of the natural extension of an arbitrary $S$-system when $S$ is a countable cancellative left reversible semigroup, which is a special case of the aforementioned work.

\begin{theorem}\label{GeneralNaturalExtension}
    If $S$ is a countable cancellative left reversible semigroup with group of right quotients $G$, and $\mathcal{X} := (X,\mathscr{B},\mu,(T_s)_{s \in S})$ is a $S$-system, then there exists a natural extension $\mathcal{Y} := (Y,\mathscr{A},\nu,(R_g)_{g \in G})$ of $\mathcal{X}$.
\end{theorem}

\begin{proof}
    Let $Y = X^G$ and let $\mathscr{A}$ be the product $\sigma$-algebra. Let $F \in \mathscr{P}_f(G)$ be arbitrary, let $\mathcal{A} := \{A_f\}_{f \in F} \subseteq \mathscr{B}$, and let $C_{\mathcal{A}} \in \mathscr{A}$ denote the corresponding cylinder set. To be more precise, we have

    \begin{equation}
        C_{\mathcal{A}} := \left\{(x_g)_{g \in G}\ |\ x_f \in A_f\ \forall\ f \in F\right\}.
    \end{equation}
    In order to define $\nu$, it suffices to define $\nu(C_\mathcal{A})$ for all $\mathcal{A}$. To this end, we recall that $S$ is a thick subset of $G$, so let $s \in S$ be such that $Fs \subseteq S$, and we provisionally define

    \begin{equation}
        \nu(C_\mathcal{A}) = \mu\left(\bigcap_{f \in F}T_{fs}^{-1}A_f\right).
    \end{equation}
    We will now show that $\nu(C_F)$ does not depend on the choice of $s$. Let $s_1,s_2 \in S$ be such that $Fs_1,Fs_2 \subseteq S$. Since $S$ is left reversible, let $t_1,t_2 \in S$ be such that $s_1t_1 = s_2t_2$. We see that

    \begin{alignat*}{2}
        \mu\left(\bigcap_{f \in F}T_{fs_1}^{-1}A_f\right) &= \mu\left(T_{t_1}^{-1}\bigcap_{f \in F}T_{fs_1}^{-1}A_f\right) = \mu\left(\bigcap_{f \in F}T_{fs_1t_1}^{-1}A_f\right)\\
        &= \mu\left(\bigcap_{f \in F}T_{fs_2t_2}^{-1}A_f\right) = \mu\left(T_{t_2}^{-1}\bigcap_{f \in F}T_{fs_2}^{-1}A_f\right) = \mu\left(\bigcap_{f \in F}T_{fs_2}^{-1}A_f\right).
    \end{alignat*}
    Now let $\{\mathcal{A}_i\}_{i = 0}^k$ be such that $C_{\mathcal{A}_0} = \bigcup_{i = 1}^KC_{\mathcal{A}_i}$.
    It can be checked that $\nu\left(C_{\mathcal{A}_0}\right) = \sum_{i = 1}^I\nu\left(C_{\mathcal{A}_i}\right)$, so by Kolmogorov's consistency theorem $\nu$ extends to a measure on $\mathscr{A}$. We now define $(R_g)_{g \in G}$ by $R_{g_0}(x_g)_{g \in G} = (x_{gg_0})_{g \in G}$.
    We will now check that $R$ is a measure preserving action of $G$ on $Y$.
    Let $F \in \mathscr{P}_f(G)$ and $g \in G$ both be arbitrary.
    Let $\mathcal{A} = \{A_f\}_{f \in F}$, let $\mathcal{A}' = \{B_t\}_{t \in Fg^{-1}}$ with $B_t = A_{tg}$, let $s_1,s_2 \in S$ be such that $Fs_2 \subseteq S$ and $Fg^{-1}s_1\cup\{s_2^{-1}g^{-1}s_1\} \subseteq S$.
    Since $\nu\left(R_g^{-1}C_{\mathcal{A}}\right) = \nu\left(C_{\mathcal{A}'}\right)$, we have

    \begin{alignat*}{2}
        \nu\left(R_g^{-1}C_{\mathcal{A}}\right)& = \mu\left(\bigcap_{f \in F}T_{fg^{-1}s_1}^{-1}A_f\right) = \mu\left(T_{s_2^{-1}g^{-1}s_1}^{-1}\bigcap_{f \in F}T_{fs_2}^{-1}A_f\right) = \mu\left(\bigcap_{f \in F}T_{fs_2}^{-1}A_f\right) = \nu\left(C_{\mathcal{A}}\right).
    \end{alignat*}

    Now that we have shown that $\mathcal{Y}$ is a $G$-system we proceed to show that $\mathcal{Y}'$ is an extension of $\mathcal{X}$. To this end, we define the measurable factor map $\pi:Y\rightarrow X$ by $\pi(x_g)_{g \in G} = x_e$. We see that for any $A \in \mathscr{B}$, we have $\pi^{-1}A = C_{\mathcal{A}}$ with $F = \{e\}$ and $\mathcal{A} = \{A\}$, so $\nu(\pi^{-1}A) = \mu(A)$. It remains to show that $\pi R_s = T_s\pi$ for all $s \in S$. Now let us fix some $s \in S$. Since $(X,\mathscr{B},\mu)$ is a standard probability space, we may pick a countable collection $\{A_n\}_{n = 1}^\infty \subseteq \mathscr{B}$ such that $\mu(A_n) > 0$ for all $n$, and for any distinct $x,y \in X$ there exists $n \ge 1$ for which $x \in A_n$ and $y \notin A_n$. Now let $s \in S$ and $g_0 \in G$ be arbitrary, let $r \in S$ be such that $\{sg_0r,g_0r\} \subseteq S$ and observe that
    
    \begin{alignat*}{2}
        &\nu\left(\left\{(x_g)_{g \in G} \in Y\ |\ T_sx_{g_0} \neq x_{sg_0}\right\}\right) \le \sum_{n = 1}^\infty\nu\left(\left\{(x_g)_{g \in G}\ |\ T_sx_{g_0} \in A_n\ \&\ x_{sg_0} \in A_n^c\right\}\right)\\
        =&\sum_{n = 1}^\infty\mu\left(T_{g_0r}^{-1}T_s^{-1}A_n\cap T_{sg_0r}^{-1}A_n^c\right) = 0.
    \end{alignat*}
    It follows that for $\nu$-a.e. $(x_g)_{g \in G} \in Y$ we have

    \begin{equation}
        \pi R_s(x_g)_{g \in G} = \pi(x_{gs})_{g \in G} = x_s = T_sx_e = T_s\pi(x_g)_{g \in G}.
    \end{equation}
    Now that we have shown $\mathcal{Y}$ to be an invertible extension of $\mathcal{X}$, it remains to show that any other invertible extension $\mathcal{Z} := (Z,\mathscr{C},\rho,(V_g)_{g \in G})$ has $\mathcal{Y}$ as a factor, and that the factor map from $\mathcal{Z}$ to $\mathcal{Y}$ is unique. 
    To this end, let $\pi_1:Z\rightarrow X$ be the factor map from $\mathcal{Z}'$ to $\mathcal{X}$. We will show that the map $\pi_2:Z\rightarrow Y$ given by $\pi_2(z) = (\pi_1(V_gz))_{g \in G}$ is a factor map from $\mathcal{Z}$ to $\mathcal{Y}$. 
    To this end, we see that for any $F \in \mathscr{P}_f(G)$ and any $\mathcal{A} = \{A_f\}_{f \in F} \in \mathscr{B}^F$ we have

    \begin{alignat*}{2}
        \rho\left(\pi_2^{-1}C_\mathcal{A}\right)& = \rho\left(\left\{z \in Z\ |\ V_fz \in \pi_1^{-1}A_f\ \forall\ f \in F\right\}\right) = \rho\left(\bigcap_{f \in F}V_f^{-1}\pi_1^{-1}A_f\right) = \rho\left(\pi_1^{-1}\bigcap_{f \in F}T_f^{-1}A_f\right)\\
        &= \mu\left(\bigcap_{f \in F}T_f^{-1}A_f\right) = \nu(C_\mathcal{A}),
    \end{alignat*}
    so $\nu = (\pi_2)_*\rho$. We see that for all $h \in G$ and $\rho$-a.e. $z \in Z$, we have

    \begin{equation}
        \pi_2(V_{h}z) = (\pi_1(V_g(V_hz)))_{g \in G} = (\pi_1(V_{gh}z))_{g \in G} = R_h(\pi_1(V_gz))_{g \in G} = R_h\pi_2(z).
    \end{equation}

    Lastly, we need to show that $\pi_2$ is unique, so let $\pi_3:Z\rightarrow Y$ be any factor map satisfying $\pi_1 = \pi\circ\pi_3$. 
    We see that for $\rho$-a.e. $z \in Z$, we have $\pi(R_g\pi_3(z)) = \pi(\pi_3(V_gz)) = \pi_1(V_gz)$ for all $g \in G$. Consequently, for $\rho$-a.e. $z \in Z$ we have $\pi_3(z) = (\pi_1(V_gz))_{g \in G} = \pi_2(z)$.
\end{proof}
\subsection{Density Ramsey theory}
A fundamental result in density Ramsey theory is the Furstenberg Correspondence principle \cite[Theorem 1.1]{FurstenbergsProofOfSzemeredi}, which allows us to use tools from analysis and dynamical systems in order to prove combinatorial results. The correspondence princple has been generalized to the setting of countable amenable groups \cite[Theorem 5.8]{CombinatorialAndDiophantineApplicationsOfErgodicTheory}, and even countable left amenable semigroups \cite[Theorem 2.1]{NoncommutativeSchur}. Here we record a version of the correspondence principle for amenable groups that is suited to our needs in later sections.

\begin{theorem}[Furstenberg correspondence principle {}]\label{FurstenbergCorrespondencePrinciple}
    Let $G$ be a countable amenable group, let $\mathcal{F} = (F_n)_{n = 1}^\infty$ be a left F\o lner sequence, let $h:G\rightarrow G$ be a homomorphism, and let $E \subseteq G$ be such that $\overline{d}_{\mathcal{F}}(E) > 0$. There exists a $G$-system $(Y,\mathscr{A},\nu,(T_g)_{g \in G})$ and an $A \in \mathscr{A}$ with $\mu(A) = \overline{d}_{\mathcal{F}}(E)$ such that for any $k \in \mathbb{N}$ and any $g_1,\cdots,g_k \in G$, one has

    \begin{equation}
        \overline{d}_{\mathcal{F}}\left(E\cap h(g_1)^{-1}E\cap\cdots\cap h(g_k)^{-1}E\right) \ge \nu\left(A\cap T_{g_1}^{-1}A\cap\cdots\cap T_{g_k}^{-1}A\right).
    \end{equation}
\end{theorem}

\begin{proof}
Since we only use facts about the Stone-\v{C}ech compactification of a semigroup in this proof and the proof of Lemma \ref{EnhancedPartitionLemmaWithHomogeneity}, we refer the reader to \cite{AlgebraInTheSCC} for background. By replacing $(F_n)_{n = 1}^\infty$ with a subsequence if needed, we will assume without loss of generality that

\begin{equation}
    \overline{d}_{\mathcal{F}}(E) = \lim_{n\rightarrow\infty}\frac{|E\cap F_n|}{|F_n|}.
\end{equation}
We let $\nu_n$ be the measure on $\beta G$ given by $\nu_n(\overline{A}) = \frac{|A\cap F_n|}{|F_n|}$ for all $A \subseteq G$, and we let $\nu$ be any weak$^*$ limit point of $\{\nu_n\}_{n = 1}^\infty$. We see that $\nu(\overline{E}) = \overline{d}_{\mathcal{F}}(E)$. For each $s \in G$, let $R_s:G\rightarrow G$ be given by $R_s(g) = h(s)g$, and let $T_g:\beta G\rightarrow\beta G$ be the unique continuous extension of $R_g$. Since $\mathcal{F}$ is a left F\o lner sequence, we see that $\nu$ is an invariant measure for each $T_g$. We let $\mathscr{A}$ be the $\sigma$-algebra generated by $\{T_g\overline{E}\}_{g \in G}$. Since $\mathscr{A}$ is countably generated, $(\beta G,\mathscr{A},\nu)$ is isomorphic to a standard probability space, so $(\beta G,\mathscr{A},\nu,(T_g)_{g \in G})$ is a $G$-system. Lastly, we see that for any $k \in \mathbb{N}$ and any $g_1,\cdots,g_k \in G$ we have

\begin{alignat*}{2}
    &\nu\left(\overline{E}\cap T_{g_1}^{-1}\overline{E}\cap\cdots\cap T_{g_k}^{-1}\overline{E}\right) = \lim_{m\rightarrow\infty}\frac{|E\cap R_{g_1}^{-1}E\cap\cdots\cap R_{g_k}^{-1}E\cap F_{n_m}|}{|F_{n_m}|}\\
    \le&\overline{d}_{\mathcal{F}}\left(E\cap h(g_1)^{-1}E\cap\cdots\cap h(g_k)^{-1}E\right).
\end{alignat*}
\end{proof}

The following result follows immediately from the work of Austin \cite{AmenableSzemeredi}, but we give a proof for the sake of completeness. For a F\o lner sequence $\mathcal{F} = (F_n)_{n = 1}^\infty$ in an amenable group $G$, we let $\mathcal{F}^d$ denote the F\o lner sequence $(F_n^d)_{n = 1}^\infty$ in $G^d$.

\begin{theorem}\label{ModifiedAmenableSzemeredi}
    Let $G$ be a countable amenable group, let $d \in \mathbb{N}$, let $\phi_1,\cdots,\phi_d:G\rightarrow G$ be endomorphisms, and let $\mathcal{F} = (F_n)_{n = 1}^\infty$ be a left F\o lner sequence in $G$. If $E \subseteq G^d$ satisfies $\overline{d}_{\mathcal{F}^d}(E) > 0$, then 

    \begin{alignat*}{2}
        \overline{d}_{\mathcal{F}}\bigg(\bigg\{&g \in G\ |\ \overline{d}_{\mathcal{F}^d}\bigg(\bigg\{ (x_1,\cdots,x_d) \in G^d\text{ such that }\left(\phi_1(g)x_1,x_2,\cdots,x_d\right),\\
        &\left(\phi_1(g)x_1,\phi_2(g)x_2,\cdots,x_d\right),\cdots,\left(\phi_1(g)x_1,\phi_2(g)x_2,\cdots,\phi_d(g)x_d\right) \in E\bigg\}\bigg) > 0\bigg\}\bigg) > 0
    \end{alignat*}
\end{theorem}

\begin{proof}
    We apply Theorem \ref{FurstenbergCorrespondencePrinciple} with $S = G^d$, the left F\o lner sequence $\mathcal{F}^d$, and the endomorphism $h(g_1,\cdots,g_d) = (\phi_1(g_1),\cdots,\phi_d(g_d))$ to obtain a $G^d$-system $(Y,\mathscr{A},\nu,(T_g)_{g \in G^d})$ and a $A \in \mathscr{A}$ such that for all $k \in \mathbb{N}$ and all $g_1,\cdots,g_k \in G^d$ we have
    
    \begin{equation}
        \overline{d}_{\mathcal{F}^d}\left(E\cap h(g_1)^{-1}E\cap\cdots\cap h(g_k)^{-1}E\right) \ge \nu\left(A\cap T_{g_1}^{-1}A\cap\cdots\cap T_{g_k}^{-1}A\right)
    \end{equation}
    For $1 \le i \le d$ and $g \in G$, we write $T_{i,g}$ to denote 

    \begin{equation}
        T_{(\underbrace{0,\cdots,0}_{i-1},g,0,\cdots,0)},
    \end{equation}
    and we observe that $\{(T_{i,g})_{g \in G}\}_{i = 1}^d$ is a family of $d$ commuting measure preserving actions of $G$ on $Y$. We now invoke \cite[Theorem B]{AmenableSzemeredi}, which tells us that 

    \begin{equation}\label{ModifyingAustinEquation}
        \lim_{n\rightarrow\infty}\frac{1}{|F_n|}\sum_{g \in F_n}\nu\left(T_{1,g}^{-1}A\cap \left(T_{1,g}^{-1}T_{2,g}^{-1}\right)A\cap\cdots\cap\left(T_{1,g}^{-1}\cdots T_{d,g}^{-1}\right)A\right) > 0.
    \end{equation}
    The desired result now follows from the observation that

    \begin{alignat*}{2}
        &(\phi_1(g)x_1,x_2,\cdots,x_d),\cdots,(\phi_1(g)x_1,\phi_2(g)x_2,\cdots,\phi_d(g)x_d) \in E\text{ if and only if}\\
        &\bigcap_{i = 1}^dh(\underbrace{0,\cdots,0}_{i-1},g,0,\cdots,0)^{-1}E \neq \emptyset
    \end{alignat*}
\end{proof}

We also need the following two mild variations of the Bergelson intersectivity lemma. The proofs of both lemmas are almost identical to those of Lemma 5.10 and Corollary 5.11 of \cite{CombinatorialAndDiophantineApplicationsOfErgodicTheory}, but we include the proofs nonetheless for the sake of completeness.

\begin{lemma}[Bergelson Intersectivity Lemma]\label{ModifiedContinuousBergelsonIntersectivityLemma}
    Let $\left(X,\mathscr{B},\mu,(T_r)_{r \in R\setminus\{0\}}\right)$ be a measure preserving system with $R\setminus\{0\}$ acting multiplicatively, and let $\mathcal{F} = (F_n)_{n = 1}^\infty$ be a F\o lner sequence in $(R,+)$. If $(A_r)_{r \in R} \subseteq \mathscr{B}$ is such that $\mu(A_r) \ge \alpha$ for some $\alpha > 0$ and all $r \in R$, then there exists $P \subseteq R$ with $\overline{d}_{\mathcal{F}}(P) > 0$, such that for all $r_1,\cdots,r_k \in P$, we have $\mu\left(\bigcap_{i = 1}^kA_{r_i}\right) > 0$.\footnote{The only difference between this lemma and \cite[Lemma 5.10]{CombinatorialAndDiophantineApplicationsOfErgodicTheory} is that the F\o lner sequence $\mathcal{F}$ is additive rather than multiplicative. However, the fact that this modification could be made was already implied by the discussion preceding \cite[Lemma 5.10]{CombinatorialAndDiophantineApplicationsOfErgodicTheory}.}
\end{lemma}

\begin{proof}
    For any finite set $F \subseteq R$, let $A_F = \bigcap_{r \in F}A_r$. Deleting, if needed, a set of measure zero from $\bigcup_{r \in R}A_r$, we assume without loss of generality that $\mu(A_F) > 0$ if and only if $A_F \neq \emptyset$. We define

    \begin{equation}
        f_n(x) = \frac{1}{|F_n|}\sum_{r \in F_n}\mathbbm{1}_{A_r}(x),
    \end{equation}
    and we observe that $0 \le f_n \le 1$ and $\int_Xf_nd\mu \ge \alpha$ for all $n \in \mathbb{N}$. Let $f(x) = \limsup_{n\rightarrow\infty}f_n(x)$. By Fatou's Lemma, we have

    \begin{equation}
        \int_Xfd\mu = \int_X\limsup_{n\rightarrow\infty}f_nd\mu \ge \limsup_{n\rightarrow\infty}\int_Xf_nd\mu \ge \alpha,
    \end{equation}
    so there exists some $x_0 \in X$ for which $f(x_0) \ge \alpha$. Consequently, we may pick a sequence $(n_i)_{i = 1}^\infty$ for which

    \begin{equation}\label{EquationForModifiedIntersectivity}
        \lim_{i\rightarrow\infty}f_{n_i}(x_0) = f(x_0) \ge \alpha.
    \end{equation}
    Let $P = \{r \in R\ |\ x_0 \in A_r\}$, and observe that Equation \eqref{EquationForModifiedIntersectivity} tells us that

    \begin{equation}
        \lim_{i\rightarrow\infty}\frac{|P\cap F_{n_i}|}{|F_{n_i}|} \ge \alpha,
    \end{equation}
    so $\overline{d}_{\mathcal{F}}(P) \ge \alpha$. Lastly, we see that for any finite $F \subseteq P$, we have $x_0 \in A_F$, hence $\mu(A_F) > 0$.
\end{proof}

\begin{lemma}\label{ModifiedDiscreteBergelsonIntersectivityLemma}
        Let $R$ be a countable integral domain, let $\mathcal{F} = (F_n)_{n = 1}^\infty$ be a F\o lner sequence in $(R,+)$, let $\Phi = (\Phi_n)_{n = 1}^\infty$ be a F\o lner sequence in $(R\setminus\{0\},\cdot)$, and let $E \subseteq R\setminus\{0\}$ be be such that $\overline{d}_{\Phi}(E) = c > 0$. Then there exists a set $P \subseteq G$ with $\overline{d}_{\mathcal{F}}(P) \ge c$, such that for any $g_1,\cdots,g_m \in P$, we have $\overline{d}_{\Phi}\left(\bigcap_{i = 1}^mg_i^{-1}E\right) > 0.$
\end{lemma}

\begin{proof}
        Let $K$ denote the field of fractions of $R$, and observe that $\Phi$ is still a F\o lner sequence in $(K^\times,\cdot)$. The Furstenberg correspondence principle gives us a measure preserving system $(X,\mathscr{B},\mu,(T_k)_{k \in K^\times})$ with $K^\times$ acting multiplicatively, and a set $A \in \mathscr{B}$ such that for all $k_1,\cdots,k_m \in K^\times$ we have

        \begin{equation}
            \overline{d}_{\Phi}\left(\bigcap_{i = 1}^mk_i^{-1}E\right) \ge \mu\left(\bigcap_{i = 1}^mT_{k_i}^{-1}A\right).
        \end{equation}
        Consequently, it suffices to first take $P' \subseteq K^\times$ as given by Lemma \ref{ModifiedContinuousBergelsonIntersectivityLemma}, and then set $P = P'\cap R$ since $\overline{d}_{\Phi}(P) = \overline{d}_{\Phi}(P')$.
\end{proof}

Given an integral domain $R$, a collection $\mathcal{A} \subseteq \mathscr{P}_f(R\setminus\{0\})$ is \textbf{multiplicatively translation invariant} if for any $A \in \mathcal{A}$ and any $r \in R\setminus\{0\}$, we have $rA \in \mathcal{A}$.

\begin{corollary}\label{CorollaryToTheIntersectivityLemma}
    Let $R$ be a countable integral domain, and let $\mathcal{A} \subseteq \mathscr{P}_f(R)$ be multiplicatively translation invariant.
    \begin{enumerate}[(i)]
        \item Let $\delta \in [0,1)$ be arbitrary. If for every $B \subseteq R$ with $d^*(B) > \delta$ there exists $A \in \mathcal{A}$ with $A \subseteq B$, then for every $B \subseteq R\setminus\{0\}$ with $d^*_\times(B) > \delta$, there exists $A \in \mathcal{A}$ with $A \subseteq B$.

        \item Let $\delta \in (0,1]$ be arbitrary. If for every $B \subseteq R$ with $d^*(B) \ge \delta$ there exists $A \in \mathcal{A}$ with $A \subseteq B$, then for every $B \subseteq R\setminus\{0\}$ with $d^*_\times(B) \ge \delta$, there exists $A \in \mathcal{A}$ with $A \subseteq B$.
    \end{enumerate}
\end{corollary}

\begin{proof}
    We only give the proof of (i) since the proof of (ii) is similar. Let $B \subseteq R\setminus\{0\}$ be such that $\overline{d}_{\Phi}(B) > \delta$ for some F\o lner sequence $\Phi$ in $(R\setminus\{0\},\cdot)$. Let $\mathcal{F}$ be an arbitrary F\o lner sequence in $(R,+)$, and let $P$ be given by Lemma \ref{ModifiedDiscreteBergelsonIntersectivityLemma}. Since $\overline{d}_{\mathcal{F}}(P) = \overline{d}_{\Phi}(B) > \delta$, let $A = \{a_i\}_{i = 1}^k \in \mathcal{A}$ be such that $A \subseteq P$. Since $\overline{d}_{\Phi}\left(\bigcap_{i = 1}^ka_i^{-1}E\right) > 0$, let $e \in \bigcap_{i = 1}^ka_i^{-1}E$ be arbitrary, and observe that $eA \subseteq E$.
\end{proof}

\begin{lemma}\label{DensityLemmaWithHomogeneity}
    Let $R$ be a countably infinite integral domain with field of fractions $K$. For any $m \in \mathbb{N}$ and any $k_1,\cdots,k_m \in K^\times$ we have the following:
    \begin{enumerate}[(i)]
        \item If $A \subseteq R$ is such that $d^*(A) > 0$, then $A$ contains a solution to the system of equations

        \begin{equation}\label{MainLinearSystemOfEquations3}
            \frac{z_{4i-3}-z_{4i-2}}{z_{4i-1}-z_{4i}} = k_i\text{ for all }1 \le i \le m.
        \end{equation}
        Furthermore, the solution can be taken such that $z_i \neq z_j$ when $i \neq j$.
        
        \item If $A \subseteq R\setminus\{0\}$ is such that $d^*_\times(A) > 0$, then $A$ contains a solution $(z_1,\cdots,z_{4m})$ to the system \eqref{MainLinearSystemOfEquations3}, such that $z_i \neq z_j$ for $i \neq j$.
    \end{enumerate}
\end{lemma}

\begin{proof}
    We first prove part (i). We proceed by induction on $m$ and we begin with the base case of $m = 1$. Let $k_1 = \frac{r_1}{s_1}$ with $r_1,s_1 \in R\setminus\{0\}$. Let $\mathcal{F}$ be any F\o lner sequence in $(R,+)$ for which $\overline{d}_{\mathcal{F}}(A) > 0$. We use Theorem \ref{ModifiedAmenableSzemeredi} with $(G,\cdot) = (R,+)$, $d = 2$, $E = A^2$, $\phi_1(g) = s_1g$, and $\phi_2(g) = r_1g$. We obtain $g \in R\setminus\{0\}$ for which

    \begin{equation}
        \overline{d}_{\mathcal{F}^2}\left(\left\{(x_1,x_2) \in R^2\ |\ (s_1g+x_1,x_2),(x_1,r_1g+x_2) \in A^2\right\}\right) > 0.
    \end{equation}
    Since $r_1,s_1,g \in R\setminus\{0\}$, we see that $x_1+s_1g \neq x_1$ and $x_2+r_1g \neq x_2$. Now let us observe that for any finite set $F \subseteq R$, we have

    \begin{alignat*}{2}
        1 \ge &\overline{d}_{\mathcal{F}^2}\left(\bigcup_{f \in F}\left\{(x,x+s_1g+f) \right\}_{x \in G}\right) = \sum_{f \in F}\overline{d}_{\mathcal{F}^2}\left(\left\{(x,x+s_1g+f) \right\}_{x \in G}\right)\\
        = &|F|\overline{d}_{\mathcal{F}^2}\left(\left\{(x,x+s_1g) \right\}_{x \in G}\right)\text{, so }\\
        0 = &\overline{d}_{\mathcal{F}^2}\left(\left\{(x,x+s_1g) \right\}_{x \in G}\right).
    \end{alignat*}
    After performing a similar calculation with $(x+r_1g,x)$, $(x,x)$, and $(x+r_1g,x+s_1g)$, we see that

    \begin{alignat*}{2}
        \overline{d}_{\mathcal{F}^2}\big(\big\{(x_1,x_2) \in R^2\ |\ & (s_1g+x_1,x_2),(x_1,r_1g+x_2) \in A^2\\
        &\text{ and }\{x_1,x_2,x_1+s_1g,x_2+r_1g\}\text{ are all distinct}\big\}\big) > 0.
    \end{alignat*}
    Lastly, we take $(z_1,z_2,z_3,z_4) = (s_1g+x_1,x_1,x_2,r_1g+x_2) \in A^4$ for which $z_i \neq z_j$ when $i \neq j$ and observe that

    \begin{equation}
        \frac{z_1-z_2}{z_3-z_4} = \frac{(r_1g+x_2)-x_2}{(s_1g+x_1)-x_1} = \frac{r_1}{s_1}.
    \end{equation}
    
    Having completed the base case, we proceed to the inductive step. Once we have picked $z_1,\cdots,z_{4m-4} \in A$, we let $A' = A\setminus\{z_i\}_{i = 1}^{4m-4}$. Since $d^*(A') = d^*(A) > 0$, we may apply the base case of the induction to pick distinct $z_{4m-3},z_{4m-2},z_{4m-1},z_{4m} \in A'$ for which

    \begin{equation}
        \frac{z_{4m-3}-z_{4m-2}}{z_{4m-1}-z_{4m}} = k_m.
    \end{equation}
    
    Part (ii) is now an immediate consequence of Corollary \ref{CorollaryToTheIntersectivityLemma}(i) applied to the collection $\mathcal{A} = \{(z_i)_{i = 1}^{4m}\ \|\ z_i \neq z_j\ \forall\ i \neq j\text{ and }(z_i)_{i = 1}^{4m}\text{ is a solution to the system in }\eqref{MainLinearSystemOfEquations3}\}$.
\end{proof}
\subsection{Partition Ramsey Theory}

\begin{definition}
    Let $R$ be an integral domain, let $S \subseteq R$, and let $p_1,\cdots,p_m \in R[x_1,\cdots,x_n]$. Given $r \in \mathbb{N}$, the system of equations $p_i(x_1,\cdots,x_n) = 0$ for all $1 \le i \le m$ is \textbf{(injectively) r-partition regular over $S$} if for every partition of the form $S = \bigcup_{i = 1}^rC_i$, there is some $1 \le i_0 \le r$ and some (distinct) $a_1,\cdots,a_n \in C_{i_0}$ such that $p_i(a_1,\cdots,a_n) = 0$ for all $1 \le i \le m$. The system of equations $p_i(x_1,\cdots,x_n) = 0$ is \textbf{(injectively) partition regular over $S$} if it is (injectively) $r$-partition regular over $S$ for every $r \in \mathbb{N}$.
\end{definition}

\begin{lemma}\label{InfinitePartitionRegularityLemma}
    Let $R$ be an integral, and suppose that $p_1,\cdots,p_m \in R[x_1,\cdots,x_n]$ are such that the system of equations
    
    \begin{equation}\label{GeneralSystemOfEquations}
        p_i(x_1,\cdots,x_n) = 0\text{ for all } 1 \le i \le m
    \end{equation}
    is partition regular over some infinite set $S \subseteq R$, and it does not possess a constant solution $x_i = c \in S$ for all $1 \le i \le n$. Then for any partition $S = \bigcup_{i = 1}^\ell C_i$, there exists some $1 \le i_0 \le \ell$ and a sequence $(a_1(t),\cdots,a_m(t))_{t = 1}^\infty \subseteq C_{i_0}^m$ such that the following hold:
    \begin{enumerate}[(i)]
        \item $p_i(a_1(t),\cdots,a_n(t)) = 0$ for all $t \in \mathbb{N}$.

        \item For $t_1 < t_2$ and any $i,j \in [1,n]$ we have $a_i(t_1) \neq a_j(t_2)$.

        \item If the system of equations in \eqref{GeneralSystemOfEquations} is injectively partition regular, then we also have that $a_i(t) \neq a_j(t)$ for all $t \in \mathbb{N}$ and all $i \neq j$.
    \end{enumerate}
\end{lemma}

\begin{proof}
    By the pigeon hole principle, it suffices to construct a sequence $(a_1(t),\cdots,a_n(t))_{t = 1}^\infty$ satisfying conditions (i)-(iii) for which we also have $a_1(t),\cdots,a_n(t) \in C_{i(t)}$ for some function $i:\mathbb{N}\rightarrow [1,\ell]$. We will construct such a sequence by induction on $t$. For the base case of $t = 1$, we may pick $i(1) \in [1,\ell]$ and (distinct) $a_1(1),\cdots,a_n(1) \in C_{i(i)}$ satisfying the system of equations in \eqref{GeneralSystemOfEquations} since we assume that the system is (injectively) partition regular over $S$. Now let us assume that $(a_1(t),\cdots,a_n(t))_{t = 1}^{T-1}$ have been constructed. Let $A_T = \{a_i(t)\ |\ 1 \le i \le n\ \&\ 1 \le t \le T-1\}$, and consider the partition $S = \left(\bigcup_{A_T}\{a\}\right)\cup\left(\bigcup_{i = 1}^\ell(C_i\setminus A_T)\right)$. Since the system of equations in \eqref{GeneralSystemOfEquations} is (injectively) partition regular over $S$, but possesses no constant solution, we may pick $i(T) \in [1,\ell]$ and (distinct) $a_1(T),\cdots,a_n(T) \in C_{i(T)}\setminus A_T$ satisfying the system of equations in \eqref{GeneralSystemOfEquations}. 
\end{proof}

\begin{remark}\label{RemarkReducingPolynomialSystemsToSinglePolynomials}
    We observe that if $R$ is an integral domain for which the field of fractions $K$ is not algebraically closed, then the question of partition regularity of a system of polynomial equations is equivalent to that of the partition regularity of a single polynomial equation. To see this, it suffices to note that if $f \in R[z]$ has no root in $K$, then for any $p_1,p_2 \in R[x_1,\cdots,x_n]$ we have

\begin{equation}
    p_1(x_1,\cdots,x_n)^{\text{deg}(f)}f\left(\frac{p_2(x_1,\cdots,x_n)}{p_1(x_1,\cdots,x_n)}\right) = 0\text{ if and only if }p_i(x_1,\cdots,x_n) = 0\text{ for } i = 1,2. 
\end{equation}
\end{remark}

We will need to use Rado's Theorem for integral domains later on, so we must first recall the columns condition. 

\begin{definition}\label{DefinitionOfColumnsCondition}
    Let $R$ be an integral domain, let $K$ denote the field of fractions of $R$, let $m,n \in \mathbb{N}$, let $\textbf{A} \in M_{m\times n}(R)$, and let $\vec{c}_i$ denote the $i^{th}$ column of $\textbf{A}$. The matrix $\textbf{A}$ satisfies the \textbf{columns condition} if there is a partition $[1,n] = \bigcup_{i = 1}^rC_i$ of the columns of $\textbf{A}$ satisfying the following conditions:
    \begin{enumerate}[(i)]
        \item $\displaystyle\sum_{i \in C_1}\vec{c}_i = 0$.

        \item For $2 \le j \le r$, $\sum_{i \in C_j}\vec{c_i}$ can be written as a $K$-linear combinatorion of $\bigcup_{i = 1}^{k-1}C_i$.
    \end{enumerate}
\end{definition}

\begin{theorem}[{\cite[Theorem A]{RadoForRingsAndModules}}]\label{RadosTheoremForIntegralDomains}
    Let $R$ be an integral domain, let $m,n \in \mathbb{N}$, and let $\textbf{A} \in M_{m\times n}(R)$. The system of equations $\textbf{A}\vec{x} = \vec{0}$ is partition regular over $R\setminus\{0\}$ if and only if $\textbf{A}$ satisfies the columns condition.
\end{theorem}

We remark that when $R = \mathbb{Z}$, Theorem \ref{RadosTheoremForIntegralDomains} is due to Rado \cite{RadosTheorem}.

\begin{corollary}\label{PartitionLemmaWithHomogeneity}
    Let $R$ be an integral domain with field of fractions $K$. For any $m \in \mathbb{N}$ and any $k_1,\cdots,k_m \in K$, the system of equations

        \begin{equation}\label{MainLinearSystemOfEquations1}
            \frac{z_{3i-2}-z_{3i-1}}{z_{3i}} = k_i\text{ for all }1 \le i \le m,
        \end{equation}
        is partition regular over $R\setminus\{0\}$. Furthermore, if $k_1,\cdots,k_m \in K^\times$, then the system of equations in \eqref{MainLinearSystemOfEquations1} is injectively partition regular over $R\setminus\{0\}$.
\end{corollary}

\begin{proof}
    Firstly, we rewrite the equations $\frac{z_{3i-2}-z_{3i-1}}{z_{3i}} = k_i = \frac{r_i}{s_i}$ as $s_iz_{3i-2}-s_iz_{3i-1}-r_iz_{3_i} = 0$. It follows that the system of equations in \eqref{MainLinearSystemOfEquations1} can be represented as $\textbf{A}\vec{z} = \vec{0}$ where 

    \begin{equation}
        \textbf{A} = \begin{pmatrix}
            s_1 & -s_1 & -r_1 & 0 & 0 & 0 & \cdots & 0 & 0 & 0\\
            0 & 0 & 0 & s_2 & -s_2 & -r_2 & \cdots & 0 & 0 & 0\\
            \vdots & \vdots &\vdots &\vdots & \vdots&\vdots & \ddots & \vdots & \vdots & \vdots\\
            0 & 0 & 0 & 0 & 0 & 0 & \cdots & s_m & -s_m & -r_m\\
        \end{pmatrix}.
    \end{equation}
    We see that $\textbf{A}$ satisfies the columns condition by taking $C_1 = \bigcup_{i = 1}^m\{3i-2,3i-1\}$ and $C_2 = \{3i\}_{i = 1}^m$, so first result now follows from Theorem \ref{RadosTheoremForIntegralDomains}.
    
    Now let us assume that $k_i \neq 0$ for $1 \le i \le m$. Let $T$ be a monotile for $K^\times$ for which $\{k_i+1,1-k_i\}_{i = 1}^m\setminus\{0\} \subseteq T$, and let $C$ be a center set for $T$. Given an arbitrary partition $R\setminus\{0\} = \bigcup_{i = 1}^\ell C_i$, we create a new partition $R\setminus\{0\} = \bigcup_{t \in T}\bigcup_{i = 1}^\ell(C_i\cap tC)$. Using Lemma \ref{InfinitePartitionRegularityLemma}, pick $1 \le i_0 \le \ell$ and $t_0 \in T$ such that $C_{i_0}\cap t_0C$ contains infinitely many solutions $(z_1(n),\cdots,z_{3m}(n))_{n = 1}^\infty$ to the system in \eqref{MainLinearSystemOfEquations1}. We see that $\frac{z_{3i-2}(n)}{z_{3i}(n)},\frac{z_{3i-1}(n)}{z_{3i}(n)} \in CC^{-1}$ for all $n \in \mathbb{N}$ and $1 \le i \le m$. Since $CC^{-1}\cap T = \{1\}$ and $z_i(n) \neq 0$, we see that we cannot have $\frac{z_{3i-2}(n)}{z_{3i}(n)} = k_i+1$ or $\frac{z_{3i-1}(n)}{z_{3i}(n)} = 1-k_i$, so we cannot have $z_{3i}(n) = z_{3i-1}(n)$ or $z_{3i}(n) = z_{3i-2}(n)$. Since $k_i \neq 0$, we see that we cannot have $z_{3i-2}(n) = z_{3i-1}(n)$. It now suffices to take $z_{3i-2} = z_{3i-2}(i)$, $z_{3i-1} = z_{3i-1}(i)$, and $z_{3i} = z_{3i}(i)$ for $1 \le i \le m$.
\end{proof}

If $(S,+)$ is a commutative cancellative semigroup, then $A \subseteq S$ is \textbf{piecewise syndetic} if there exists a $F \in \mathscr{P}_f(S)$ for which $\bigcup_{f \in F}(A-f)$ is thick.
We recall that if $B \subseteq S$ is thick, then $d^*(B) = 1$, so if $A$ is piecewise syndetic, then $d^*(A) > 0$. It is well known that for any finite partition $S = \bigcup_{i = 1}^\ell C_i$, at least one of the cells $C_i$ is piecewise syndetic.
When we are working with an integral domain $R$, we say that $A \subseteq R$ is \textbf{additively piecewise syndetic} if it is a piecewise syndetic subset of $(R,+)$, and $A$ is \textbf{multiplicatively piecewise syndetic} if it is a piecewise syndetic subset of $(R\setminus\{0\},\cdot)$.
The following result is a special case of \cite[Theorem 27]{ax+by=cwmzn}.

\begin{lemma}\label{TranslationInvariancePartitionTheorem}
    Let $R$ be an integral domain and let $\mathcal{A} \subseteq \mathscr{P}_f(R\setminus\{0\})$ be multiplicatively translation invariant. The collection $\mathcal{A}$ is partition regular\footnote{A collection $\mathcal{A} \subseteq \mathscr{P}_f(R\setminus\{0\})$ is \textbf{partition regular} if for any finite partition $R\setminus\{0\} = \bigcup_{i = 1}^\ell C_i$, there exists $A \in \mathcal{A}$ and $1 \le i_0 \le \ell$ for which $A \subseteq C_{i_0}$.} if and only if for any multiplicatively piecewise syndetic set $B \subseteq R\setminus\{0\}$, there exists $A \in \mathcal{A}$ with $A \subseteq B$.
\end{lemma}

\begin{lemma}\label{EnhancedPartitionLemmaWithHomogeneity}
    Let $R$ be an integral domain, let $F \subseteq R$ not be multiplicatively piecewise syndetic, and let $m \in \mathbb{N}$. For any partition of the form $R = \bigcup_{i = 1}^\ell C_i$, there exists a $1 \le i_0 \le \ell$ such that the following holds:
    \begin{enumerate}[(i)]
        \item There exists $y_1,y_2 \in C_{i_0}$ with $\frac{y_1}{y_2} \in R\setminus F$.

        \item For any $k_1,\cdots,k_m \in K$, chosen after $y_1$ and $y_2$ are fixed, there exists $z_1,\cdots,z_{3m} \in C_{i_0}$ for which

        \begin{equation}\label{MainLinearSystemOfEquations2}
            \frac{z_{3i-2}-z_{3i-1}}{z_{3i}} = k_i\text{ for all }1 \le i \le m.
        \end{equation}
        Furthermore, if $k_1,\cdots,k_m \in K^\times$, then we can take the $z_i$ to be distinct.
    \end{enumerate}
\end{lemma}

\begin{proof}
    For the proof of this lemma we will assume familiarity with ultrafilter methods in partition Ramsey theory. For a quick introduction to these methods, the reader is refered to \cite[Section 9]{ax+by=cwmzn}. Theorem 27 of \cite{ax+by=cwmzn} allows us to pick a non-principal ultrafilter $p \in \beta R^*$ satisfying the following two properties:
    \begin{enumerate}
        \item For every $r \in R\setminus\{0\}$, we have $rR \in p$.

        \item If $A \in p$, then $A$ is multiplicatively piecewise syndetic.

        \item For every $A \in p$, $A$ contains a solution to the system of equations in \eqref{MainLinearSystemOfEquations2}.

        \item If $k_i \neq 0$ for $1 \le i \le m$, then for every $A \in p$, $A$ contains an injective solution to the system of equations in \eqref{MainLinearSystemOfEquations2}.
    \end{enumerate}
    Since $R = \bigcup_{i = 1}^\ell C_i$ is a finite partition, we may pick $1 \le i_0 \le \ell$ such that $C_{i_0} \in p$. Let $y_2 \in C_{i_0}$ be arbitrary. Since $F$ is not multiplicatively piecewise syndetic, neither is $y_2F$, so $y_2F \notin p$. Since $y_2R \in p$, we see that $(y_2R\cap C_{i_0})\setminus y_2F \in p$, so let $y_1 \in (y_2R\cap C_{i_0})\setminus y_2F$ be arbitrary. It is now clear that part (i) is satisfied. Now that $y_1$ and $y_2$ are fixed, we are given $k_1,\cdots,k_m \in K$. We use property 3 (property 4) of our ultrafilter $p$ to pick (distinct) $z_1,\cdots,z_{3m} \in C_{i_0}$ satisfying Equation \eqref{MainLinearSystemOfEquations2}.
    
\end{proof}

We now record two special cases of the compactness principle in Ramsey theory. For a more general treatment, the reader is referred to \cite[Chapter 1]{GRSRamseyTheory}.

\begin{theorem}[the compactness principle]\label{PartitionCompactnessPrinciple}
    Let $R$ be an integral domain, let $S \subseteq R$, and let $p \in R[x_1,\cdots,x_n]$.
    \begin{enumerate}[(i)]
        \item Given $r \in \mathbb{N}$, the equation $p(x_1,\cdots,x_n) = 0$ is (injectively) $r$-partition regular over $S$ if and only if it also (injectively) $r$-partition regular over some finite set $F_r \subseteq S$.

        \item The equation $p(x_1,\cdots,x_n) = 0$ is (injectively) partition regular over $S$ if and only if for every $r \in \mathbb{N}$ it is (injectively) $r$-partition regular over some finite set $F_r \subseteq S$.
    \end{enumerate}
     
\end{theorem}

\subsection{Descriptive set theory}

Our main results are Theorems \ref{MainResultForSigma01} and \ref{MainResultForPi02}, and they determine the descriptive complexity of sets of polynomials that are of interest in Ramsey theory. For example, we show that for many integral domains $R$, the set of polynomials $p \in R[x_1,\cdots,x_n]$ for which $p(x_1,\cdots,x_n) = 0$ is partition regular over $R\setminus\{0\}$ is $\Pi_2^0$-complete, hence undecidable. In the case of $R = \mathbb{Z}$, we obtain this result conditional upon Hilbert's 10th problem for $\mathbb{Q}$, and for the other domains we obtain the result unconditionally. Consequently, we give here a brief review of the relevant notions of complexity. 

A set $A\subseteq \N^k$ is \textbf{computable}, 
or $\Delta^0_1$, if there is an algorithm which 
given input $\vec n=(n_1,\dots,n_k)$ will terminate in a finite number of steps with the answer as to 
whether or not $\vec n \in A$. A set $A\subseteq \N^k$ is \textbf{semi-computable}, or $\Sigma^0_1$, 
if there is a $B\subseteq \N^{k+1}$ which is $\Delta^0_1$ such that for all 
$\vec n=(n_1,\dots,n_k)$ we have $\vec n \in A$ if and only if there exists $m \in \mathbb{N}$ such that $ 
(n_1,\dots,n_k,m)\in B$.
Also, the existential quantifier over $m$ can be replaced by a finite string of existential quantifiers over $\N$
in the definition. That is, if  for all $\vec n$ we have $(n_1,\dots,n_k)\in A$ iff 
there exists $(m_1,\cdots,m_\ell) \in \mathbb{N}^\ell$ such that $(n_1,\dots,n_k,m_1,\dots,m_\ell)\in B$, where $B \subseteq \mathbb{N}^{k+\ell}$ is $\Delta^0_1$, then $A$ is $\Sigma^0_1$.

The $\Sigma^0_1$ sets are also known as the \textbf{computably enumerable (c.e.)} sets as they are 
the sets for which an algorithm can list the members of the set. A set $A\subseteq \N^k$ is $\Pi^0_1$ (or \textbf{co-c.e.})
if the complement $\N^k \setminus A$ is $\Sigma^0_1$. It is a standard fact that $\Delta^0_1=
\Sigma^0_1\cap \Pi^0_1$, that is, a set is computable iff both the set and its complement are 
$\Sigma^0_1$. 

The collections $\Delta^0_1$, $\Sigma^0_1$, $\Pi^0_1$ are the first few levels of the \textbf{lightface 
hierarchy} of sets of integers, which measures the complexity of such sets. The higher levels 
$\Delta^0_n$, $\Sigma^0_n$, $\Pi^0_n$ are defined as follows. $\Sigma^0_{n}=\exists^{\N} \Pi^0_{n-1}$. 
That is, $A\subseteq \N^k$ is $\Sigma^0_{n}$ if there is $B\subseteq \N^{k+1}$, with $B$ a $\Pi^0_{n-1}$ set,
such that $(n_1,\dots,n_k)\in A$ iff there exists $m \in \mathbb{N}$ such that $(n_1,\dots,n_k,m)\in B$. Again, it does not affect the definition
if we allow a finite string of existential quantifiers instead of a single existential quantifier. 
We then let $\Pi^0_n$ be the collection of sets whose complements are $\Sigma^0_n$, which is abbreviated 
$\Pi^0_n= \check{\Sigma}^0_n$. Finally, we set $\Delta^0_n=\Sigma^0_n \cap \Pi^0_n$, that is, a set 
$A\subseteq \N^k$ is $\Delta^0_{n}$  iff it is both a $\Sigma^0_n$ and a $\Pi^0_n$ set, or equivalently,
$A$ and $\N^k\setminus A$ are both $\Sigma^0_n$. 

A set $A\subseteq \mathbb{N}$ is said to be $\Sigma^0_n$-\textbf{universal} if every $\Sigma^0_n$ set $B\subseteq \mathbb{N}$
is {\bf computably reducible} to $A$. By this we mean that there is a computable function $f\colon \mathbb{N}\to \mathbb{N}$ 
such that $n \in B$ iff $f(n)\in A$ for any $n\in \mathbb{N}$. The function $f$ is called a reduction of $B$ to $A$. 
A set $A\subseteq \mathbb{N}$ is $\Sigma^0_n$-{\bf complete} if $A\in \Sigma^0_n$ and $A\notin \Pi^0_n$. 
Every $\Sigma^0_n$-universal set is $\Sigma^0_n$-complete.

Similarly,  we define the notion of $\Pi^0_n$-complete.

The collections $\Delta^0_n$, $\Sigma^0_n$, $\Pi^0_n$ form the lightface arithmetical hierarchy 
of subsets of integers. We are concerned in this paper with sets of integers (or things which can be coded by
integers such as polynomials), but the lightface hierarchy generalizes to a hierarchy of subsets of any reasonable 
(more precisely, recursively presented) Polish space $X$. 

A countably infinite integral domain $R$ is \textbf{computable} if there is a bijection $\phi:R\rightarrow\mathbb{N}$ such that the maps $+_\phi:\mathbb{N}^2\rightarrow\mathbb{N}$ and $\cdot_\phi:\mathbb{N}^2\rightarrow\mathbb{N}$ given by $n_1+_\phi n_2 = \phi(\phi^{-1}(n_1)+\phi^{-1}(n_2))$ and $n_1\cdot_\phi n_2 = \phi(\phi^{-1}(n_1)\cdot\phi^{-1}(n_2))$ respectively are computable, i.e., $n_1+_\phi n_2$ and $n_1\cdot_\phi n_2$ can each be calculated in a finite number of steps. 

We will also use the following notation that is common in logic. Given a set $S$, a subset $B \subseteq S$, and an element $m \in S$, we write $B(m)$ to abbreviate $m \in B$. Similarly, if $B \subseteq S\times S$ and $m,n \in S$, then we write $B(m,n)$ to abbreviate $(m,n) \in B$.

\subsection{Hilbert's 10th Problem}\label{Hilberts10thProblemSubsection}

Hilbert's 10th problem asked whether there was an algorithm which, given a polynomial $p\in \mathbb{Z}[x_1,\dots,x_n]$ with integer coefficients in some finite number of variables, would decide if the polynomial has an integer root, that is, a tuple $(a_1,\dots,a_n)\in \mathbb{Z}^n$ with $p(A_1,\dots,a_n)=0$. 
In 1970, Matiyasevich \cite{SolutionToH10}, drawing on earlier work of Davis, Putnam, and Robinson, solved Hilbert's 10th problem by showing that there was no such algorithm, which is frequently phrased by saying the question is undecidable.
More precisely, it was shown that the set of polynomials with coefficients in $\mathbb{Z}$ which have integer roots is a $\Sigma^0_1$ complete set (identifying polynomials with integers via some reasonable coding). 
We refer the reader to \cite{Hilberts10thProblemIsUnsolvable} for an exposition of the solution to Hilbert's 10th problem, as well as a discussion of the history.
This seminal result led to much further investigation into the corresponding question over other integral domains, and also allowing the roots to lie in extensions of the base ring. 
In particular one important generalization is Hilbert's 10th problem for $\mathbb{Q}$. 
This asks whether there is an algorithm for deciding if a polynomial with integer coefficients (or equivalently with rational coefficients) has a root in $\mathbb{Q}$. This problem remains open. 
For other integral domains, the analog of Hilbert's 10th problem has been settled and the question has been similarly shown to be undecidable.  

In this paper we consider the complexity of partition regularity and related notions, both for polynomials over $\mathbb{Z}$ and over other computable integral domains.
A compactness argument (Theorem \ref{PartitionCompactnessPrinciple}
for partition regularity and Theorem~\ref{DensityCompactnessPrinciple} for density regularity) shows that the set polynomials $p$ for which the equation $p(x_1,\cdots,x_n) = 0$ is partition regular, is a $\Pi^0_2$ set (and likewise for density regularity). Our main results, Theorems~\ref{MainResultForPi02}
and \ref{MainResultForSigma01}, show that the set of partition regular polynomials is in fact $\Pi^0_2$-complete, and if we fix the number of colors
$\ell$ of the partitions, then the set is $\Sigma^0_1$-complete. 
Our proofs will involve reductions based on Hilbert's 10th problem for the quotient field ($\mathbb{Q}$ in the case of polynomials over $\mathbb{Z}$),
and thus depend on the $\Sigma^0_1$-completeness of the set of polynomials 
having a zero in the quotient field. For polynomials over $\mathbb{Z}$, this makes our theorem conditional on the undecidabilty of Hilbert's 10th problem for $\mathbb{Q}$, which as we noted is still open. For other integral domains, the corresponding result is known, and so we get an unconditional result. 

In all known cases of arithmetic interest where Hilbert's 10th problem for the integral domain $R$ (abbreviated HTP$(R)$) has been shown to be undecidable,
the proof actually gives a reduction of an arbitrary $\Sigma^0_1$ set $A\subseteq \N^k$ to the set of polynomials 
over $R$ having a root in $R$. That is, there is a computable function which assigns to each 
$\vec a=(a_1,\dots,a_k)\in \N^k$ a polynomial $p_{\vec a}\in R[y_1,\dots,y_\ell]$ such that 
$\vec a \in A$ iff $p_{\vec a}$ has a root in $R$. Since this equivalence holds in all known 
cases where HTP$(R)$ is known to be undecidable, we will henceforth take the phrase
``HTP$(R)$ is undecidable'' to mean that any $\Sigma^0_1$ subset of $\N^k$ can be computably reduced
to the set of polynomials over $R$ having a root in $R$. Similarly, we take the phrase ``HTP$(R\setminus\{0\})$ is undecidable'' to mean that any $\Sigma^0_1$ subset of $\N^k$ can be computably reduced to the set of polynomials over $R$ having a root in $R\setminus\{0\}$.
Lemma \ref{LemmaForHTPVariations} tells us that if HTP$(R)$ is undecidable then HTP$(R\setminus\{0\})$, is also undecidable.

Denef \cite{denef1978diophantine} showed that HTP$(K)$ is undecidable when $K = F(t)$ where $F$ is a totally real field. Pheidas \cite{pheidas1991hilbert} showed that HTP$(F(t))$ is undecidable when $F$ is a finite field of odd characteristic, and Videla \cite{videla1994hilbert} showed that HTP$(F(t))$ is undecidable when $F$ is a finite field of even characteristic. Later Shlapentokh \cite{shlapentokh1996diophantine} showed that HTP$(K)$ is undecidable when $K$ is an algebraic function field over a finite field of characteristic greater than two. More recently, Eisentr\"{a}ger and Shlapentokh \cite{eisentrager2017hilbert} showed that if $K$ is a countable function field that does not contain the algebraic closure of a finite field, then HTP$(K$) is undecidable.

In the proof of Hilbert's 10th problem for $\mathbb{Z}$, given any $\Sigma^0_1$ set $A\subseteq \mathbb{N}$ 
a computable function $f\colon \mathbb{N} \to$ polynomials is constructed so that for all $n$, $n \in A$ 
iff $f(n)$ has a root in $\mathbb{Z}$. Some simple variations of this will also be useful which we state in the next lemma.

\begin{lemma}\label{LemmaForHTPVariations}
Let $R$ be a computable, infinite integral domain. 
Let $P_0$ be the set of polynomials $p$ over $R$ 
which have a root in $R$. 
Let $P_1$ be the set of polynomials $p$ over $R$ 
which have a root in $R\setminus \{ 0\}$. If $P_0$ is $\Sigma^0_1$-complete, then $P_1$ is also 
$\Sigma^0_1$-complete. 
\end{lemma}

\begin{proof}
To see this, let $p(x_1,\dots,x_n)$ be a polynomial in
$\bigcup_{n = 1}^\infty \mathbb{Z}[x_1,\cdots,x_n]$, and let $p'$
be the polynomial $p'(y_1,\dots,y_n,z_1,\dots,z_n)= 
p((y_1+z_1),\dots,(y_n+z_n))$. 
Then $p$ has a root in $R$ if and only if $p'$ has a root in $R$ in which none of the variables are equal to $0$. 
It is clear that if $p'$ has a root in $R\setminus\{0\}$, then $p$ has a root in $R$, so let us now assume that $p$ has a root $(a_1,\dots,a_n)$ in $R$. Since $R$ is infinite, for each $1 \le i \le n$, we may pick and arbitrary $y_i \in R\setminus\{0,a_i\}$ and then let $z_i = a_i-y_i$.
Thus the map $p \mapsto p'$ is a computable reduction of $P_0$ to $P_1$. 
\end{proof}

We will be concerned with the $\Pi^0_2$ completeness of various sets, and our proofs wil make use of the following 
definition.

\begin{definition} \label{def:mp}
Let $R$ be an infinite computable integral domain with field of fractions $K$. 
Let $S\subseteq \N\times \Z$ be a $\Sigma^0_1$ set. A triple $(p,f,g)$ consisting of a polynomial $p\in R[x,y,z_1,\dots,z_k]$, a computable map $g:K\rightarrow\mathbb{Z}$ for which $g(0) = 0$ and $g$ is a surjective homomorphism from $(K^\times,\cdot)$ to $(\mathbb{Z},+)$, and a computable map $f:\N\to R$ is an \textbf{additive master polynomial} for $S$ 
if for all $m,n \in \N$ we have the following:
\begin{enumerate}
    \item $S(m,n)$ if and only if for every $b \in K$ with $g(b)=n$, the polynomial  $p(f(m), b,z_1,\dots,z_k)$
has a root $z_1,\cdots,z_k \in K^\times$.

    \item $\neg S(m,n)$ if and only if for every $b \in K$ with $g(b)=n$, the polynomial  $p(f(m), b,z_1,\dots,z_k)$
does not have a root $z_1,\cdots,z_k \in K$.
\end{enumerate}
If $(p,f,g)$ and $S$ are as above, but $g$ is a surjective homomorphism from $(K^\times,\cdot)$ to $(\mathbb{Q}^\times,\cdot)$ (instead of $(\mathbb{Z},+)$), and $S \subseteq \mathbb{N}\times\mathbb{Q}$ (instead of $S \subseteq \mathbb{N}\times\mathbb{Z}$), then $(p,f,g)$ is a \textbf{multiplicative master polynomial} for $S$.
The domain $R$ \textbf{admits master polynomials} if one of the following holds:
\begin{enumerate}[(i)]
\item Every $\Sigma_1^0$ set $S \subseteq \mathbb{N}\times\mathbb{Z}$ has a an additive master polynomial.\footnote{This is equivalent to saying that a universal $\Sigma_1^0$ set $S \subseteq \mathbb{N}\times\mathbb{Z}$ admits an additive master polynomial.}

\item Every $\Sigma_1^0$ set $S \subseteq \mathbb{N}\times\mathbb{Q}$ has a multiplicative master polynomial.\footnote{This is equivalent to saying that a universal $\Sigma_1^0$ set $S \subseteq \mathbb{N}\times\mathbb{Q}$ admits an additive master polynomial.}
\end{enumerate}
\end{definition}

The following property of multiplicative homomorphisms will be of use in Section 4. 

\begin{lemma}\label{MultiplicativeHomomorphismLemma}
    If $(G,+)$ is a countably infinite group and $g:(K^\times,\cdot)\rightarrow(G,+)$ is a surjective homomorphism, then $d^*_\times\left(g^{-1}(\{0\})\right) = 0$. 
\end{lemma}

\begin{proof}
    For each $h \in G$ let $r_h \in K$ be such that $g(r_h) = h$. 
    Let $\mathcal{F}$ be a F\o lner sequence in $(K^\times,\cdot)$ for which $d_{\mathcal{F}}(g^{-1}(\{0\})) = d^*_\times(g^{-1}(\{0\}))$. 
    Then $d_{\mathcal{F}}(g^{-1}(\{0\})) = d_{\mathcal{F}}(r_hg^{-1}(\{0\}))$ for all $h \in G$, and $r_{h_1}g^{-1}(\{0\})\cap r_{h_2}g^{-1}(\{0\}) = \emptyset$ for all distinct $h_1,h_2 \in G$, so $d_{\mathcal{F}}(g^{-1}(\{0\})) = 0$.
\end{proof}

Given a $\Sigma^0_1$ set $S\subseteq \N\times \N$, we can effectively compute from $S$ a $\Sigma^0_1$
set $S'\subseteq \N\times \Z$ with the following properties:
\begin{enumerate}
    \item For all $m \in \mathbb{N}$ we have $\neg S'(m,0)$.
    
    \item For all $m \in \N$, we have for all $n \in \mathbb{N}$ $S(m,n)$ if and only if 
for all $n \in \mathbb{Z}\setminus\{0\}$ $S'(m,n)$.

    \item If $\exists n_0\neq 0\ \neg S'(m,n_0)$ then $\neg S'(m,n)$ holds for all $|n| \ge |n_0|$.
\end{enumerate}
Namely, we can take $S'(m,n)\leftrightarrow (n\neq 0) \wedge  \forall p\leq |n|\ S(m,p)$. Thus, given a $\Pi^0_2$ set $A\subseteq \N$
we may assume $A(m)\leftrightarrow \forall n \neq 0\ S(m,n)$ where $S$ is $\Sigma^0_1$ and has the property
that for any $m$ if $\exists n \neq 0\ \neg S(m,n)$ then $\neg S(m,n)$ holds for cofinitely many $n$. $S$ also has the property that $\neg S(m,0)$ for all $m \in \mathbb{N}$.

Now let $\phi:\mathbb{Q}\rightarrow\mathbb{Z}$ be any computable bijection for which $\phi(0) = 0$, $\phi(q) > 0$ if $q > 0$, and $\phi(-q) = -\phi(q)$, then define the order $\le_\phi$ on $\mathbb{Q}$ by $q_1 \le_\phi q_2$ if and only if $\phi(q_1) \le \phi(q_2)$. For any $\Sigma_1^0$ set $S \subseteq \mathbb{N}\times\mathbb{N}$, we can effectively compute from $S$ a $\Sigma_1^0$ set $S'' \subseteq \mathbb{N}\times\mathbb{Q}$ with the following properties:

\begin{enumerate}
    \item For all $m \in \mathbb{N}$ we have $\neg S''(m,0)$.
    
    \item For all $m \in \N$, we have for all $n \in \mathbb{N}$ $S(m,n)$ if and only if 
for all $n \in \mathbb{Q}\setminus\{0\}$ $S''(m,n)$.

    \item If $\exists n_0\neq 0\ \neg S''(m,n_0)$ then $\neg S'(m,n)$ holds for all $|n| \ge_\phi |n_0|$.
\end{enumerate}
Namely, we can take $S''(m,n)\leftrightarrow S'(m,\phi(n))$. This will be a convenient starting point for our later arguments.

The additive/multiplicative master polynomial translates $\Sigma^0_1$ statments over $\N$ into Diophantine
relations over the ring or field in question. In all cases of interest for this paper we have that
master polynomials exist. We must first review some terminology in order to discuss these cases.

We use $\mathbb{F}_q$ to denote the finite field of $q$ elements. A \textbf{rational function field} over a base field $F$ is a field of the form $F(t_1,\cdots,t_n)$. 
The \textbf{ring of integers} of a rational function field is the polynomial ring $F[t_1,\cdots,t_n]$. 
An \textbf{algebraic function field} over a base field $F$ is a finite dimensional algebraic extension of a rational function field over $F$, i.e., a field of the form $F(t_1,\cdots,t_n)(\alpha)$ where $p(\alpha) = 0$ for some $p(x) \in F(t_1,\cdots,t_n)[x]$. 
The \textbf{ring of integers} of an algebraic function field $K = F(t_1,\cdots,t_n)(\alpha)$ is $R := \left\{y \in K\ |\ p(y) = 0\text{ for some monic }p(x) \in F(t_1,\cdots,t_n)[x]\right\}$. 
It is worth noting that if $K$ is an algebraic function field of the form $\mathbb{F}_q(t)(\alpha)$, then the ring of integers $R$ is isomorphic to $\mathbb{F}_{q^\ell}[t]$ for some $\ell \in \mathbb{N}$. A \textbf{prime} $\mathfrak{p}$ of an algebraic function field $K$ is an irreducible element of the ring of integers $R$.

\begin{lemma}
Let $R$ be a computable integral domain with field of fractions $K$.
\begin{enumerate}
\item
If we assume HTP$(\Q)$, then $\Z$ admits master polynomials.
\item
If $K=\mathbb{F}_q(t)$, then $R = \mathbb{F}_q[t]$ admits master polynomials. More generally, if $K$ is an algebraic function field over a finite field of characteristic greater than $2$, then the ring of integers $R$ admits master polynomials.
\end{enumerate}
\end{lemma}

\begin{proof}
In the case $R=\Z$, $K=\Q$, we let $f$ and $g$ be the identity functions. 
Then Definition~\ref{def:mp} requires that for all $m,n \in \N$ that 
$S(m,n)$ iff $p(m,n,z_1,\dots,z_k)$ has a root in $\Q$. This is precisely the statement of HTP$(\Q)$, that is, 
the statement that we can computably reduce the $\Sigma^0_1$ set $S$ to the set of polynomials
in $\Z[x_1,\dots,x_k]$ having a root in $\Q$.

In the cases in which $K$ is one of the fields mentioned in 2,
the function $g$ in the definition of master polynomial is given by 
$g(x)=\text{ord}_{\mathfrak{p}}(x)$ for some prime $\mathfrak{p}$ of the ring of integers $R$. 
Since $\text{ord}_{\mathfrak{p}}(0)$ is undefined, we define $g(0) = 0$.
The results of \cite{pheidas1991hilbert,videla1994hilbert,shlapentokh1996diophantine} construct master polynomials using this $g$. The function 
$f$ in the definition of master polynomial is any computable function such that $f(n)$
satisfies $\text{ord}_{\mathfrak{p}}(f(n))=n$. 
\end{proof}

\section{Compactness and uniformity principles for density Ramsey theory}

Let $(S,\cdot)$ be a semigroup. A collection $\mathcal{A} \subseteq \mathscr{P}_f(S)$ is \textbf{right translation invariant} if for any $A \in \mathcal{A}$ and any $s \in S$, we have $As \in \mathcal{A}$. In the future sections, we will be considering $\mathcal{A}$ in two situations. The first situation is when $(S,\cdot) = (R,+)$ for a countable integral domain $R$, $p \in R[x_1,\cdots,x_n]$ is a polynomial satisfying $p(x_1,\cdots,x_n) = p(x_1+r,\cdots,x_n+r)$ for all $r \in R$, and $\mathcal{A}$ is the zero-set of $p$. The second situation is when $(S,\cdot) = (R,\cdot)$, the polynomial $p$ is homogeneous, and $\mathcal{A}$ is again the zero-set of $p$. While we will only need the density Ramsey theory compactness principle for cancellative abelian semigroups, we choose to also prove it for cancellative left amenable semigroups since the extra level of generality does not yield any additional difficulties and is of independent interest.

\begin{theorem}[Compactness Principle]\label{DensityCompactnessPrinciple}
    Let $S$ be a countably infinite cancellative left amenable semigroup, let $\delta \in (0,1]$, and let $\mathcal{A} \subseteq \mathscr{P}_f(S)$ be right translation invariant. The following are equivalent:
    \begin{enumerate}[(i)]
        \item\label{UpperBanachDensityNonstrict} If $B \subseteq S$ satisfies $d^*(B) \ge \delta$, then there exists $A \in \mathcal{A}$ with $A \subseteq B$.

        \item\label{UBDFinitization} There exists $K \in \mathscr{P}_f(S)$ and $\epsilon > 0$ such that for any $(K,\epsilon)$-invariant set $F \in \mathscr{P}_f(S)$ and for all $B \subseteq F$ with $|B| \ge \delta|F|$, there exists $A \in \mathcal{A}$ with $A \subseteq B$.

        \item\label{FiniteSetsCharacterization} There exists a $H \in \mathscr{P}_f(S)$ such that for all $B \subseteq H$ with $|B| \ge \delta|H|$, there exists $A \in \mathcal{A}$ with $A \subseteq B$.

        \item\label{FiniteSetsCharacterization2} There exists a $H \in \mathscr{P}_f(S)$ and a $\epsilon > 0$ such that for all $B \subseteq H$ with $|B| > (\delta-\epsilon)|H|$, there exists $A \in \mathcal{A}$ with $A \subseteq B$.
        \item\label{UpperBanachDensityNonstrict2} There exists $\epsilon > 0$ such that for all $B \subseteq S$ satisfying $d^*(B) > \delta-\epsilon$, there exists $A \in \mathcal{A}$ with $A \subseteq B$.
    \end{enumerate}
\end{theorem}

\begin{proof}
    Throughout this proof we let $G$ be the group of right quotients of $S$.

    We begin by showing that \eqref{UpperBanachDensityNonstrict}$\rightarrow$\eqref{UBDFinitization}. Our proof for this direction is motivated by the proof of uniformity of recurrence given in \cite[Section 5.4]{RigidityAndWeakMixingInAbelianGroups} (see also \cite[Proposition 1.3]{SetsOfRecurrenceAndGeneralizedPolynomials}). We will show the contrapositive, so let us assume that for any $K \in \mathscr{P}_f(S)$ and any $\epsilon > 0$, there exists a $(K,\epsilon)$-invariant set $F \in \mathscr{P}_f(S)$ and there exists $B \subseteq F$ with $|B| \ge \delta|F|$, such that there is no $A \in \mathcal{A}$ with $A \subseteq B$. Let us fix an exhaustion $\{e\} \subseteq K_1 \subseteq K_2 \subseteq \cdots \subseteq G$ of $G$ by finite sets, and let us fix a decreasing sequence of positive real numbers $(\epsilon_k)_{k = 1}^\infty$ with $\epsilon_1 < \delta$.

    Using Theorem \ref{CongruentTilingTheorem}, let $(\mathcal{T}_k)_{k = 1}^\infty$ be a congruent sequence of tilings of $G$ for which each shape $T \in \mathcal{S}(\mathcal{T}_k)$ is $(K_k,\epsilon_k)$-invariant. For each $T \in \mathcal{S}(\mathcal{T}_k)$ let $\mathcal{B}_T = \{B \subseteq T\ |\ |B| > (\delta-\epsilon_k)|T|\text{ and for all }A \in \mathcal{A}, A \not\subseteq B\}$. We create an infinite graph $(V,E)$ as follows. Let $V_k$ denote the union of all $\mathcal{B}_T$ with $T \in \mathcal{S}(\mathcal{T}_k)$, let $V_0 = \{\emptyset\}$, and let $V = \bigcup_{k = 0}^\infty V_k$. We produce an edge $(B_1,B_2) \in E$ if and only if there exists a $g \in G$ and a $k \in \mathbb{N}_0$ for which $B_1 \in V_k, B_2 \in V_{k+1}$, and $B_1 \subseteq B_2g$. Firstly, we will show that for all $k \in \mathbb{N}$ we have $V_k \neq \emptyset$. 
    
    To see that this is the case, let $U_k = \bigcup_{T \in \mathcal{S}(\mathcal{T}_k)}T$, let $F_k$ be a $\left(U_kU_k^{-1},\epsilon_k|U_kU_k^{-1}|^{-1}\right)$-invariant set, and let $B_k \subseteq F_k$ be such that $|B_k| > \delta|F_k|$ and $B_k$ contains no $A \in \mathcal{A}$. Lemma \ref{GoodTilingLemma} tells us that $\mathcal{T}_k$ $\epsilon_k$-tiles $F_k$, so let $V_K$ denote a union of tiles of $\mathcal{T}_k$ for which we have $V_k \subseteq F_k$ and $|F_K\setminus V_k| < \epsilon_k|F_k|$. Writing $T \in V_k$ to denote a tile $T$ of $\mathcal{T}_k$ that is contained in $V_k$, we see that

    \begin{alignat*}{2}
        \sum_{T \in V_k}|B_k\cap T| > |B_k\cap F_k| - \epsilon_k|F_k| > (\delta-\epsilon_k)|F_k| \ge \sum_{T \in V_k}(\delta-\epsilon_k)|T|,
    \end{alignat*}
    so there exists some $T \in V_k$ for which $|B_k\cap T| > (\delta-\epsilon_k)|T|$. Writing $T = T_kc$ for some $T_k \in \mathcal{S}(\mathcal{T}_k)$ and $c \in G$, we see that $|B_kc^{-1}\cap T_k| > (\delta-\epsilon_k)|T_k|$. Furthermore, since $B_k$ does not contain any $A \in \mathcal{A}$, and $\mathcal{A}$ is right translation invariant, we see that $B_kc^{-1}\cap T_k$ does not contain any $A \in \mathcal{A}$, so $B_kc^{-1} \in V_k$.
    
    Now we want to show that $(V,E)$ is a connected graph. It suffices to show that for all $k \in \mathbb{N}$, each vertex in $V_k$ is connected to a vertex in $V_{k-1}$. It is clear that all vertices in $V_1$ are connected to $V_0$, so let us now assume that $k > 1$. Let $T_k \in \mathcal{S}(\mathcal{T}_k)$ and $B \in \mathcal{B}_{T_k}$ both be arbitrary. Since $(\mathcal{T}_k)_{k = 1}^\infty$ is a congruent sequence of tilings, let $V_k$ denote the set of tiles of $\mathcal{T}_{k-1}$ for which $T_k = \bigcup_{T \in V_k}T$. We see that

    \begin{equation}
        \sum_{T \in V_k}|B\cap T| = |B| \ge (\delta-\epsilon_k)|T_k| \ge \sum_{T \in V_k}(\delta-\epsilon_{k-1})|T|,
    \end{equation}
    so there exists some $T \in V_k$ for which $|B\cap T| \ge (\delta-\epsilon_{k-1})|T|$. Letting $T = T_{k-1}c$ for some $T_{k-1} \in \mathcal{S}(\mathcal{T}_{k-1})$ and $c \in G$, we see that $|Bc^{-1}\cap T_{k-1}| \ge (\delta-\epsilon_{k-1})|T_{k-1}|$. Since $B$ does not contain any $A \in \mathcal{A}$, and $\mathcal{A}$ is right translation invariant, we see that $B' := Bc^{-1}\cap T_{k-1}$ does not contain any $A \in \mathcal{A}$, hence $B' \in V_{k-1}$ and $(B',B) \in E$.

    Now that we have shown $(V,E)$ to be a connected graph, we also observe that it is locally finite. Indeed, for any $k \in \mathbb{N}$, there are only finitely many vertices in $V_{k-1}\cup V_{k+1}$, and each vertex in $V_k$ can only have edges going to $V_{k-1}\cup V_{k+1}$. K\H{o}nigs Lemma (see, e.g. \cite[Lemma 8.1.2]{GraphTheoryByDiestel}) tells us that there exists an infinite sequence $(B_k)_{k = 1}^\infty$ for which $B_k \in V_k$ and $(B_k,B_{k+1}) \in E$. Let $g_k$ be such that $B_k \subseteq B_{k+1}g_k$, and let $g_k' \in S$ be such that $(\{g_k\}\cup T_{k+1}g_k)g_k' \subseteq S$, where $T_{k+1} \in \mathcal{S}(\mathcal{T}_{k+1})$ is such that $B_{k+1} \in \mathcal{B}_{T_{k+1}}$. Let $s_k = g_kg_k'g_{k-1}g_{k-1}'\cdots g_1g_1'$, then observe that $B_ks_{k-1} \subseteq B_{k+1}s_k \subseteq T_{k+1}s_k \subseteq S$. We define $B = \bigcup_{k = 1}^\infty B_ks_{k-1}$, and we observe that $B$ does not contain any $A \in \mathcal{A}$. Furthermore, $\mathcal{F} = (T_ks_{k-1})_{k = 1}^\infty$ is a left F\o lner sequence in $S$, and 

    \begin{equation}
        \overline{d}_{\mathcal{F}}(B) \ge \limsup_{k\rightarrow\infty}\frac{|B_ks_{k-1}|}{|T_ks_{k-1}|} \ge \limsup_{k\rightarrow\infty}(\delta-\epsilon_k) = \delta.
    \end{equation}

    It is clear that \eqref{UBDFinitization}$\rightarrow$\eqref{FiniteSetsCharacterization}, so we proceed to show that \eqref{FiniteSetsCharacterization}$\rightarrow$\eqref{FiniteSetsCharacterization2}. We consider the cases in which $\delta$ is rational and irrational separately. If $\delta$ is irrational, then we let $K = \lfloor\delta|H|\rfloor$, and we let $\epsilon = \delta-\frac{K}{|H|}$. We now see that if $B \subseteq H$ is such that $|B| > (\delta-\epsilon)|H| = K$, then $|B| \ge K+1 > \delta|H|$, so the desired result follows. If $\delta$ is rational, then we will assume without loss of generality that $K := \delta|H| \in \mathbb{N}$, otherwise we would repeat the previous proof. Now let $g \in G\setminus H$ be arbitrary, let $H' = H\cup\{g\}$, and let $\epsilon = \delta-\frac{K}{|H|+1}$. We see that if $B \subseteq H'$ with $|B| > (\delta-\epsilon)|H'| = K$, then $|B| \ge K+1$, so $|B\cap H| \ge K = \delta|H|$, which yields the desired result.
    
    It is clear that \eqref{UpperBanachDensityNonstrict2}$\rightarrow$\eqref{UpperBanachDensityNonstrict}, so it only remains to show that \eqref{FiniteSetsCharacterization2}$\rightarrow$\eqref{UpperBanachDensityNonstrict2}. Let $H \in \mathscr{P}_f(S)$ and $\epsilon > 0$ be such that for any $B \subseteq H$ with $|B| > (\delta-\epsilon)|F|$, there exists $A \in \mathcal{A}$ with $A \subseteq B$. Given $C \subseteq S$ with $d^*(C) > \delta-\frac{1}{2}\epsilon$, we may use Lemma \ref{AlternativeCharacterizationOfUBD} to pick a $s \in S \subseteq G$ for which $|H\cap Cs^{-1}| = |Hs\cap C| \ge (\delta-\epsilon)|H|$. Consequently, there exists $A \in \mathcal{A}$ for which $A \subseteq H\cap Cs^{-1} \subseteq Cs^{-1}$, and since $\mathcal{A}$ is right translation invariant, we see that $As \in \mathcal{A}$ and $As \subseteq C$.
\end{proof}

\begin{remark}\label{RemarkJustifyingAssumptionOfTranslationInvariance}
    We observe that in Theorem \ref{DensityCompactnessPrinciple}, the assumption that $\mathcal{A}$ is right translation invariant was only used in the proofs of (i)$\rightarrow$(ii) and (iv)$\rightarrow$(v). To see that this assumption was necessary for the proof of (iv)$\rightarrow$(v), it suffices to consider $\mathcal{A} = \{s_0\}$ where $s_0 \in S$ is arbitrary. We will now give an example to show that right translation invariance is also necessary for (i)$\rightarrow$(ii).

    Let $\mathcal{F} = (F_n)_{n = 1}^\infty$ be a left F\o lner sequence in $S$ for which $F_n\cap F_m = \emptyset$ when $n \neq m$. Let $\mathcal{A}$ be the collection of finite subsets $F$ of $S$ for which $F \not\subseteq F_m$ for any $m$. We see that if $B \subseteq S$ is an infinite set (which is certainly the case if $d^*(B) > 0$), then there exists $A \in \mathcal{A}$ for which $A \subseteq B$. However, for each $m \ge 1$, the set $F_m$ does not contain any member of $\mathcal{A}$.
\end{remark}

\begin{corollary}\label{DensityRamseyTheoryCompactnessPrincipleCorollary}
    Let $S$ be a countably infinite cancellative left amenable semigroup, let $\delta \in [0,1)$, and let $\mathcal{A} \subseteq \mathscr{P}_f(X)$ be right translation invariant. The following are equivalent:
    \begin{enumerate}[(i)]
        \item If $B \subseteq S$ satisfies $d^*(B) > \delta$, then there exists $A \in \mathcal{A}$ with $A \subseteq B$.

        \item For all $\gamma > \delta$, there exists $K \in \mathscr{P}_f(S)$ and $\epsilon > 0$ such that for any $(K,\epsilon)$-invariant set $F \in \mathscr{P}_f(S)$ and for all $B \subseteq F$ with $|B| \ge \gamma|F|$, there exists $A \in \mathcal{A}$ with $A \subseteq B$.

        \item For all $\gamma > \delta$ there exists $F \in \mathscr{P}_f(S)$ such that for all $B \subseteq F$ with $|B| \ge \gamma|F|$, there exists $A \in \mathcal{A}$ with $A \subseteq B$.
    \end{enumerate}
\end{corollary}

\begin{corollary}\label{TranslationInvarianceImpliesDensityRegularity}
    Let $S$ be a countably infinite cancellative left amenable semigroup and let $\mathcal{A} \subseteq \mathscr{P}_f(S)$ be right translation invariant. There exists $\delta \in [0,1)$ for which the following hold:
    \begin{enumerate}[(i)]
        \item If $B \subseteq S$ satisfies $d^*(B) > \delta$, then there exists $A \in \mathcal{A}$ with $A \subseteq B$.

        \item There exists a $B \subseteq S$ with $d^*(B) = \delta$, such that $B$ does not contain any member of $\mathcal{A}$.
    \end{enumerate}
\end{corollary}

\begin{proof}
    It is well known that $B \subseteq S$ is thick if and only if $d^*(B) = 1$. We refer the reader to \cite{ThicksSetsAndFolnerDensity1} for a treatment of this fact in full generality. Since $\mathcal{A}$ is right translation invariant, it is immediate that any thick set contains a member of $\mathcal{A}$. Now consider

    \begin{equation}
        H = \{\alpha \ge 0\ |\ \forall\ B \subseteq S\text{ with }d^*(B) \ge \alpha,\text{ there exists }A \in \mathcal{A}\text{ such that }A \subseteq B\},
    \end{equation}
    and observe that $1 \in H$. Let $\delta = \inf(H)$. It is immediate that (i) is satisfied. If $\delta = 0$, then we see that (ii) is also satisfied since $d^*(\emptyset) = 0$. Now let us assume for the sake of contradiction that $\delta > 0$ and (ii) is not satisfied. In particular, we have that for every $B \subseteq S$ with $d^*(B) = \delta$, there exists $A \in \mathcal{A}$ for which $A \subseteq B$. Using the equivalence of parts (i) and (v) of Theorem \ref{DensityCompactnessPrinciple}, we see that that is some $\epsilon > 0$ for which $\inf(H) \le \delta-\epsilon$, which yields the desired contradiction.
\end{proof}

In special cases such as $(S,\cdot) = (\mathbb{N},+)$, particular F\o lner sequences such as $\mathcal{F} = ([1,N])_{N = 1}^\infty$ are of interest. This motivates our next result.

\begin{theorem}\label{7Equivalences}
    Let $S$ be a countably infinite cancellative left amenable semigroup, let $G$ be the group of right quotients of $S$, let $\mathcal{F} = (F_n)_{n = 1}^\infty$ be a left F\o lner sequence in $S$, let $\delta \in [0,1)$ be arbitrary, and let $\mathcal{A} \subseteq \mathscr{P}_f(S)$ be arbitrary. We have (i)$\rightarrow$(ii)$\rightarrow$(iii), and that (iii)$-$(vii) are equivalent. 
    If $\mathcal{A}$ is right translation invariant, then (i)-(vii) are all equivalent.
    \begin{enumerate}[(i)]
        \item If $B \subseteq S$ satisfies $d^*(B) > \delta$, then there exists $A \in \mathcal{A}$ with $A \subseteq B$.
        
        \item If $B \subseteq S$ satisfies $\overline{d}_{\mathcal{F}}(B) > \delta$, then there exists $A \in \mathcal{A}$ with $A \subseteq B$.

        \item For any ergodic $G$-system $(X,\mathscr{B},\mu,(T_s)_{s \in G})$ and any $B \in \mathscr{B}$ with $\mu(B) > \delta$, there exists $A \in \mathcal{A}$ for which $\mu\left(\bigcap_{a \in A}T_a^{-1}B\right) > 0$.

        \item For any $G$-system $(X,\mathscr{B},\mu,(T_s)_{s \in G})$ and any $B \in \mathscr{B}$ with $\mu(B) > \delta$, there exists $A \in \mathcal{A}$ for which $\mu\left(\bigcap_{a \in A}T_a^{-1}B\right) > 0$.
        
        \item For any $S$-system $(X,\mathscr{B},\mu,(T_s)_{s \in S})$ and any $B \in \mathscr{B}$ with $\mu(B) > \delta$, there exists $A \in \mathcal{A}$ for which $\mu\left(\bigcap_{a \in A}T_a^{-1}B\right) > 0$.

        \item If $h:G\rightarrow G$ is a homomorphism\footnote{We recall that Lemma \ref{ExtendingHomomorphismsLemma} tells us that any endomorphism of $S$ extends to an endomorphism of $G$, but there exist endomorphisms of $G$ that do not restrict to endomorphism of $S$.} and $B \subseteq S$ satisfies $d^*(B) > \delta$, then there exists $A \in \mathcal{A}$ with $d^*\left(\bigcap_{a \in A}h(a)^{-1}B\right) > 0$.

        \item If $h:G\rightarrow G$ is a homomorphism and $B \subseteq S$ satisfies $\overline{d}_{\mathcal{F}}(B) > \delta$, then there exists $A \in \mathcal{A}$ with $\overline{d}_{\mathcal{F}}\left(\bigcap_{a \in A}h(a)^{-1}B\right) > 0$.
    \end{enumerate}
    Furthermore, for $\delta \in (0,1]$ the analogous results hold when the assumptions that $\mu(B),d^*(B),\overline{d}_{\mathcal{F}}(B)\allowbreak > \delta$ are replaced by $\mu(B),d^*(B),\allowbreak\overline{d}_{\mathcal{F}}(B) \ge \delta$.
\end{theorem}

\begin{proof}
    It is clear that (i)$\rightarrow$(ii), so we proceed to show that (ii)$\rightarrow$(iii). By replacing $\mathcal{F} = (F_n)_{n = 1}^\infty$ with a subsequence, we will assume without loss of generality that $\mathcal{F}$ is tempered. Now let $(X,\mathscr{B},\mu,(T_g)_{g \in G})$ be an ergodic $G$-system and let $B \in \mathscr{B}$ be such that $\mu(B) > \delta$. Since $\mathcal{A}$ only contains countably many sets, we delete a set of measure $0$ from $B$ so that for all $A \in \mathcal{A}$ we have

    \begin{equation}
        \mu\left(\bigcap_{a \in A}T_a^{-1}B\right) = 0\text{ if and only if }\bigcap_{a \in A}T_a^{-1}B = \emptyset.
    \end{equation}
    We now invoke Theorem \ref{PointwiseErgodicTheorem} to pick some $x \in X$ such that for any $A \in \mathcal{A}$ we have

    \begin{equation}
        \lim_{N\rightarrow\infty}\frac{1}{|F_N|}\sum_{g \in F_N}\prod_{a \in A}\mathbbm{1}_B(T_{ag}x) = \mu\left(\bigcap_{a \in A}T_a^{-1}B\right).
    \end{equation}
    It follows that $E := \{g \in G\ |\ T_gx \in B\}$ satisfies $\overline{d}_{\mathcal{F}}(E) = \mu(B) > \delta$, so let $A \in \mathcal{A}$ be such that $A \subseteq E$. Since $E_A := \bigcap_{a \in A}a^{-1}E \neq \emptyset$, we see that for any $g \in E_A$ we have

    \begin{equation}
        T_gx \in \bigcap_{a \in A}T_a^{-1}B\text{, hence } \mu\left(\bigcap_{a \in A}T_a^{-1}B\right) > 0.
    \end{equation}
    
    To show that (iii)$\rightarrow$(iv), we let $\mathcal{X} := (X,\mathscr{B},\mu,(T_g)_{g \in G})$ be an arbitrary $G$-system, and we let $(Y,\mathscr{A},\nu)$ and $\{\mu_y\}_{y \in Y}$ be as in Theorem \ref{ErgodicDecomposition}. Let $Y' \in \mathscr{A}$ be given by $Y' := \{y \in Y\ |\ \mu_y(B) > \delta\}$, and observe that
    
    \begin{equation}
        \delta < \mu(B) = \int_Y\mu_y(B)d\nu(y) \le \nu(Y')+\delta(1-\nu(Y'))\text{, hence }\nu(Y') > 0.
    \end{equation}
    For each $y \in Y'$, let $A_y \in \mathcal{A}$ be such that

    \begin{equation}
        \mu_y\left(\bigcap_{a \in A_y}T_a^{-1}B\right) > 0.
    \end{equation}
    For $A \in \mathcal{A}$, let $Y_A = \{y \in Y'\ |\ A_y = A\}$. Since $\mathcal{A}$ is countable and $Y' = \bigcup_{A \in \mathcal{A}}Y_A$, let $A \in \mathcal{A}$ be such that $\nu(Y_A) > 0$. We now see that

    \begin{equation}
        \mu\left(\bigcap_{a \in A}T_a^{-1}B\right) = \int_Y\mu_y\left(\bigcap_{a \in A}T_a^{-1}B\right)d\nu(y) \ge \int_{Y_A}\mu_y\left(\bigcap_{a \in A}T_a^{-1}B\right)d\nu(y) > 0.
    \end{equation}
    
    To show that (iv)$\rightarrow$(v), we let $(Y,\mathscr{A},\nu,(R_s)_{s \in G})$ be the natural extension of $(X,\mathscr{B},\mu,(T_s)_{s \in S})$, and we let $\pi:Y\rightarrow X$ be the factor map. Since $\nu(\pi^{-1}B) = \mu(B) > \delta$, we pick $A \in \mathcal{A}$ for which

    \begin{equation}
        \mu\left(\bigcap_{a \in A}T_a^{-1}B\right) = \nu\left(\bigcap_{a \in A}R_a^{-1}\pi^{-1}B\right) > 0.
    \end{equation}

    To show that (v)$\rightarrow$(vi) we let $\Phi = (\Phi_n)_{n = 1}^\infty$ be a left F\o lner sequence for which $d^*(B) = \overline{d}_{\Phi}(B)$, and then we invoke Theorem \ref{FurstenbergCorrespondencePrinciple} to pick a $G$-system $(X,\mathscr{B},\mu,(T_s)_{s \in G})$ and a $C \in \mathscr{B}$ such that for any $k \in \mathbb{N}$ and any $g_1,\cdots,g_k \in G$, we have

    \begin{equation}
        \overline{d}_{\Phi}\left(h(g_1)^{-1}B\cap\cdots\cap h(g_k)^{-1}B\right) \ge \mu\left(T_{g_1}^{-1}C\cap\cdots\cap T_{g_k}^{-1}C\right).
    \end{equation}
    The desired result follows after observing that any $G$-system restricts to a $S$-system.
    
    The proof that (v)$\rightarrow$(vii) is the same as the proof that (v)$\rightarrow$(vi).
    The proofs that (vi) and (vii) each imply (iii) is the same as the proof that (ii)$\rightarrow$(iii).
    To show that (vi)$\rightarrow$(i) and that (vii)$\rightarrow$(ii) when $\mathcal{A}$ is right translation invariant, it suffices to let $h$ be the identity map and observe that if $b \in \bigcap_{a \in A}a^{-1}B$, then $Ab \subseteq B$.
\end{proof}

\begin{definition}\label{DefinitionOfDensityRegularity}
    Let $S$ be a countably infinite cancellative left amenable semigroup and let $\mathcal{A} \subseteq \mathscr{P}_f(S)$ be right translation invariant. Given $\delta \in [0,1)$, the collection $\mathcal{A}$ is \textbf{weakly $\delta$-density regular} if it satisfies any of the equivalent conditions appearing in Theorem \ref{7Equivalences}. If $\mathcal{A}$ is weakly $0$-density regular, then we say that $\mathcal{A}$ is \textbf{density regular}. Given $\delta \in [0,1]$, the collection $\mathcal{A}$ is \textbf{$\delta$-density regular} if for any $B \subseteq S$ satisfying $d^*(B) \ge \delta$, there exists $A \in \mathcal{A}$ for which $A \subseteq B$.
\end{definition}

Corollary \ref{TranslationInvarianceImpliesDensityRegularity} tells us that any right translation invariant collection is weakly $\delta$-density regular but not $\delta$-density regular for some $\delta \in [0,1)$.

Our next goal is to prove a uniformity principle for $\delta$-density regular collections. In order to give some context, let us recall the principle of uniformity of recurrence. Given a countably infinite semigroup $S$, a set $R \subseteq S$ is a \textbf{set of (measurable) recurrence} if for any $S$-system $(X,\mathscr{B},\mu,(T_s)_{s \in S})$ and any $A \in \mathscr{B}$ with $\mu(A) > 0$, there exists $r \in R$ for which $\mu\left(A\cap T_r^{-1}A\right) > 0$. To connect this definition to the previous discussion when $S$ is cancellative and left amenable, we observe that there is a bijection $\phi$ between right translation invariant collections $\mathcal{A} \subseteq S^2 \subseteq \mathscr{P}_f(S)$ and subsets $R \subseteq S$ given by $\phi(\mathcal{A}) := \{s \in S\ |\ \exists (r,t) \in \mathcal{A}\text{ with }sr=t\}$ and $\phi^{-1}(R) := \{(r,t) \in S^2\ |\ \exists s \in R\text{ with }sr=t\}$, and that $R$ is a set of recurrence if and only if $\mathcal{A} := \phi^{-1}(R)$ is density regular. The principle of uniformity of recurrence says that if $R$ is a set of recurrence, then for each $\delta > 0$ there is a $\gamma > 0$ and a finite set $R_\delta \subseteq R$ such that for any $S$-system $(X,\mathscr{B},\mu,(T_s)_{s \in S})$ and any $A \in \mathscr{B}$ with $\mu(A) \ge \delta$ there exists $r \in R_{\delta}$ for which $\mu\left(A\cap T_r^{-1}A\right) \ge \gamma$. In the set up of the present paper, the principle of uniformity of recurrence (roughly speaking) asserts that if the right translation invariant family $\mathcal{A} \subseteq S^2$ is density regular, then for each $\delta > 0$, there is a finite subcollection $\mathcal{A}_\delta \subseteq \mathcal{A}$ that is $\delta$-density regular. 

The uniformity of recurrence was proven for $\mathbb{Z}$ and even for an arbitrary countable group $G$ by Forrest as \cite[Theorem 2.1]{ForrestThesis} and \cite[Lemma 6.4]{ForrestThesis} respectively. 
Alternate proofs in the case of $\mathbb{Z}$ appears as \cite[Proposition 1.3]{SetsOfRecurrenceAndGeneralizedPolynomials} and \cite[Theorem 1.5]{fish2024inverse}, and the first of these proofs was extended to the case of countable abelian groups in \cite[Section 5]{RigidityAndWeakMixingInAbelianGroups}. The uniformity principle for Szemer\'edi's Theorem is classical (see, e.g. \cite[Chapter 2.5]{GRSRamseyTheory}), and the uniformity principle for the polynomial Szemer\'edi Theorem is well known in the folklore.\footnote{While Bergelson and Leibman \cite{PolynomialSzemeredi} do not mention any finitistic version of their results, Gowers \cite[Page 38]{gowers2000arithmetic} attributes such results to them anyways.} We also mention that \cite[Theorem A.2]{frantzikinakis2009powers} can be seen as a principle of uniformity of multiple recurrence for $\mathbb{Z}$. Our next result implies all of the previously mentioned uniformity principles, although it only implies \cite[Lemma 6.4]{ForrestThesis} when the group $G$ is amenable.
 
 \begin{theorem}[Uniformity of Density Regularity]\label{Uniformity}
     Let $S$ be a countably infinite cancellative left amenable semigroup, let $\delta \in (0,1]$ be arbitrary, and let $\mathcal{A} \subseteq \mathscr{P}_f(S)$ be right translation invariant and $\delta$-density regular. Then there exists a finite subcollection $\mathcal{A}' \subseteq \mathcal{A}$ and a $\gamma > 0$ such that the following hold:
     \begin{enumerate}[(i)]
         \item There exists $K \in \mathscr{P}_f(S)$ and $\epsilon > 0$ such that for any $(K,\epsilon)$-invariant set $F \in \mathscr{P}_f(S)$, and any $B \subseteq F$ with $|B| \ge (\delta-\epsilon)|F|$, there exists $A \in \mathcal{A}'$ for which $|F\cap\bigcap_{a \in A}a^{-1}B| > \gamma|F|$.

         \item If $\mathcal{F}$ is a left F\o lner sequence in $S$, and $B \subseteq S$ satisfies $\overline{d}_{\mathcal{F}}(B) \ge \delta$, then there exists $A \in \mathcal{A}'$ for which $\overline{d}_{\mathcal{F}}(\bigcap_{a \in A}a^{-1}B) \ge \gamma$.
     
         \item If $B \subseteq S$ satisfies $d^*(B) \ge \delta$, then there exists $A \in \mathcal{A}'$ for which $d^*(\bigcap_{a \in A}a^{-1}B) \ge \gamma$.

         \item For any $S$-system $\mathcal{X} := (X,\mathscr{B},\mu,(T_s)_{s \in S})$ and any $B \in \mathscr{B}$ with $\mu(B) \ge \delta$, there exists $A \in \mathcal{A}'$ for which $\mu\left(\bigcap_{a \in A}T_a^{-1}B\right) \ge \gamma$.
     \end{enumerate}
 \end{theorem}

 \begin{proof}
     We will first prove (i), and we remark that the (ii) and (iii) follow immediately from (i). Since $\mathcal{A}$ is right translation invariant and $\delta$-density regular, we may use Theorem \ref{DensityCompactnessPrinciple} to pick some $K' \in \mathscr{P}_f(S)$ and $\epsilon_1 \in \left(0,\frac{1}{2}\right)$ such that for any $(K',\epsilon_1)$-invariant set $F \in \mathscr{P}_f(S)$ and all $B \subseteq F$ with $|B| \ge (\delta-\frac{1}{2}\epsilon_1)|F|$, there exists $A \in \mathcal{A}$ with $A \subseteq B$. Let $G$ be the group of right quotients of $S$, and using Theorem \ref{CongruentTilingTheorem} let $\mathcal{T}$ be a tiling of $G$ for which each shape $\{T_i\}_{i = 1}^I$ of $\mathcal{T}$ is $(K',\epsilon_1)$-invariant. Let $T = \bigcup_{i = 1}^IT_i$ and let $\mathcal{A}' = \{A \in \mathcal{A}\ |\ A \subseteq T\}$. 
     Now let $F \in \mathscr{P}_f(S)$ be $(T,\frac{1}{8}\epsilon_1|TT^{-1}|^{-1})$-invariant, so it will also be $(TT^{-1},\frac{1}{4}\epsilon_1|TT^{-1}|^{-1})$-invariant. Lemma \ref{GoodTilingLemma} tells us that $\mathcal{T}$ $\frac{1}{4}\epsilon_1$-tiles $F$. Let $\{W_i\}_{i = 1}^M$ be a collection of tiles of $\mathcal{T}$ for which $W := \bigcup_{i = 1}^MW_i$ satisfies $W \subseteq F$ and $|F\setminus W| < \frac{1}{4}\epsilon_1|F|$. Let $B' \subseteq F$ be such that $|B'| \ge \delta|F|$, and observe that $B := B'\cap W$ satisfies $|B| \ge \left(\delta-\frac{1}{4}\epsilon_1\right)|W|$. Let $R_1 = \{1 \le i \le M\ |\ |B\cap W_i| \ge \left(\delta-\frac{1}{2}\epsilon_1\right)|W_i|\}$ and let $R_2 = [1,M]\setminus R_1$. We will assume without loss of generality that $|T_1| \le |T_2| \le \cdots \le |T_I|$, and we let $M' = |T_I|/|T_1|$. We observe that

     \begin{alignat*}{2}
        &\left(\delta-\frac{1}{4}\epsilon_1\right)|W| \le |B| = \sum_{i = 1}^M|B\cap W_i| \le \sum_{i \in R_1}|W_i|+\left(\delta-\frac{1}{2}\epsilon_1\right)\sum_{i \in R_2}|W_i|\\
        = &\left(\delta-\frac{1}{2}\epsilon_1\right)|W|+\left(1-\delta+\frac{1}{2}\epsilon_1\right)\sum_{i \in R_1}|W_i| \le \left(\delta-\frac{1}{2}\epsilon_1\right)|W|+\left(1-\delta+\frac{1}{2}\epsilon_1\right)|R_1||T_I|\text{, hence }\\
        |R_1| &\ge \frac{\frac{1}{4}\epsilon_1|W|}{\left(1-\delta+\frac{1}{2}\epsilon_1\right)|T_I|}.
     \end{alignat*}
     For each $r \in R_1$, let $W_r = T_{i_r}s_r$ and let $A_r \in \mathcal{A}'$ be such that $A_rs_r \subseteq B\cap T_{i_r}s_r$, and observe that $s_r \in W_r\cap\bigcap_{a \in A_r}a^{-1}(B\cap W_r)$. Let $A \in \mathcal{A}'$ be such that $A_r = A$ for at least $|R_1|/|\mathcal{A}'|$ values of $r$, and observe that

     \begin{equation}
         |F\cap\bigcap_{a \in A}a^{-1}B| \ge \frac{|R_2|}{|\mathcal{A}'|} \ge \frac{\frac{1}{4}\epsilon_1|W|}{\left(1-\delta+\frac{1}{2}\epsilon_1\right)|\mathcal{A}'|\cdot|T_I|} \ge \frac{\epsilon_1|F|}{8\left(1-\delta+\frac{1}{2}\epsilon_1\right)2^{|T|}|T_I|}.
     \end{equation}
     We see that it suffices to take 

     \begin{equation}
         \gamma = \frac{\epsilon_1}{8\left(1-\delta+\frac{1}{2}\epsilon_1\right)2^{|T|}|T_I|}, \epsilon = \frac{1}{8}\epsilon_1|TT^{-1}|^{-1},
     \end{equation}
    and to let $K \in \mathscr{P}_f(S)$ be such that $KK^{-1} \supseteq TT^{-1}$.
    
    We now proceed to prove (iv). Firstly, we recall the converse to the Furstenberg Correspondence Principle, which recently appeared as \cite[Theorem 1.16]{SaulsvdC} (see also \cite[Theorem 1.2]{farhangi2024asymptotic}). In particular, \cite[Theorem 1.6]{SaulsvdC} states that if $(Y,\mathscr{A},\nu,(R_g)_{g \in G})$ is a $G$-system, and $\Phi$ is a left F\o lner sequence in $G$, then for any $A \in \mathscr{A}$ there exists $B \subseteq G$ satisfying

    \begin{equation}
        \nu\left(\bigcap_{f \in F}T_f^{-1}A\right) = d_{\Phi}\left(\bigcap_{f \in F}f^{-1}B\right),
    \end{equation}
    for all $F \in \mathscr{P}_f(G)$. We now see that (iv) follows from (ii) after applying the converse to the Furstenberg Correspondence Principle with $\Phi = \mathcal{F}$ to the invertible extension of $\mathcal{X}$.
 \end{proof}

\section{Reduction to Hilbert's Tenth Problem}

Let $R$ be a computable integral domain with field of fractions $K$. For $p \in \bigcup_{n = 
 1}^\infty R[x_1,\cdots,x_n]$ let $\mathcal{A}_p \subseteq \mathscr{P}_f(R)$ denote the collection of zeroes of $p$. We remark that we are identify each root of the polynomial $p$ with a finite set in $R$ rather than a point in $R^n$. For example, if $p(x,y) = x-y$, then $\mathcal{A}_p = \{r\}_{r \in R}$. We let $\mathcal{A}_p' = \{(a_1,\cdots,a_n) \in \mathcal{A}_p\ |\ a_i \neq a_j\ \forall\ i \neq j\}$. Let $H_R \subseteq \bigcup_{n 
 = 1}^\infty R[x_1,\cdots,x_n]$ denote the set of homogeneous polynomials, i.e., $H_R = \bigcup_{n = 1}^\infty\{p \in R[x_1,\cdots,x_n]\ |\ p(rx_1,\cdots,rx_n) = r^{\text{deg}(p)}p(x_1,\cdots,x_n)\ \forall\ r \in R\}$. Let $T_R \subseteq R[x_1,\cdots,x_n]$ denote the set of additively translation invariant polynomials, i.e., $T_r = \bigcup_{n = 1}^\infty\{p \in R[x_1,\cdots,x_n]\ |\ p(x_1+r,\cdots,x_n+r) = p(x_1,\cdots,x_n)\ \forall\ r \in R\}$. For any $p \in \bigcup_{n 
 = 1}^\infty R[x_1,\cdots,x_n]$. For $\delta > 0$, let $IMDR_R(\delta)$ denote the set of polynomials $p(x_1,\cdots,x_n)$ for which every $B \subseteq R$ with $d^*_\times(B) \ge \delta$ contains an injective root of $p$, i.e., there exists $A \in \mathcal{A}_p'$ for which $A \subseteq B$. 
 For $\delta > 0$, let $IADR_R(\delta)$ denote the set of polynomials for which every $B \subseteq R$ with $d^*(B) \ge \delta$ contains an injective root of $p$. 
 Let $IMDR_R = \bigcap_{\delta > 0}IMDR_R(\delta)$ and let $IADR_R = \bigcap_{\delta > 0}IADR_R(\delta)$. We remark that IM(A)DR is an abbreviation for injectively multiplicative (additive) density regular. Let $(IPR_R(\ell))\ PR_R(\ell)$ denote the set of polynomials $p$ for which the equation $p(x_1,\cdots,x_n) = 0$ is (injectively) $\ell$-partition regular over $R\setminus\{0\}$. Let $PR_R := \bigcap_{\ell = 1}^\infty P_R(\ell)$ and let $IPR_R := \bigcap_{\ell = 1}^\infty IPR_R(\ell)$. 
 We observe that $p \in PR_R$ if and only if the equation $p(x_1,\cdots,x_n) = 0$ is partition regular over $R\setminus\{0\}$, and that $p \in IPR_R$ if and only if the equation $p(x_1,\cdots,x_n) = 0$ is injectively partition regular over $R\setminus\{0\}$. 
 Similarly, we see that $p \in IADR_R\ (p \in IMDR_R)$ if and only if for every $B \subseteq R$ with $d^*(B) > 0\ (d^*_\times(B) > 0)$, $B$ contains an injective root of $p$.

\begin{proposition}\label{DescriptiveComplexityOfRamseyTheory}
    Let $R$ be a countably infinite computable integral domain.
    \begin{enumerate}[(i)]
        \item For $\delta > 0$, the sets $IMDR_R(\delta)\cap H_R$ and $IADR_R(\delta)\cap T_R$ are $\Sigma_1^0$.

        \item For $\ell \in \mathbb{N}$, the sets $PR_R(\ell)$ and $IPR_R(\ell)$ are $\Sigma_1^0$.
        
        \item  The sets $IMDR_R\cap H_R, IADR_R\cap T_R, PR_R,$ and $IPR_R$ are $\Pi_2^0$.
    \end{enumerate}
\end{proposition}

\begin{proof}
    Since the integral domain $R$ is computable, we may assume without loss of generality that we are working over the set $\mathbb{N}$ endowed with an integral domain structure in which addition and multiplication are computable. Theorem \ref{DensityCompactnessPrinciple} tells us that $p \in IMDR_R(\delta)\cap H_R\ (p \in IADR_R(\delta)\cap T_R)$ if and only if $p \in H_R\ (p \in T_R)$ and there exists a finite set $F \subseteq R\setminus\{0\}$ such that for any $B \subseteq F$ with $|B| \ge \delta|F|$, there eixsts $A \in \mathcal{A}_p'$ with $A \subseteq B$. Since $H_R$ and $T_R$ are computable sets, it is computable to find all $B \subseteq F$ for which $|B| \ge \delta|F|$, and it is computable to check if there exists $A \in \mathcal{A}_p$ for which $A \subseteq B$, we have part $(i)$. Part (ii) follows from Theorem \ref{PartitionCompactnessPrinciple}(i) and the fact that it is computable to find all partitions of a finite set $F$ into $\ell$ pieces. Part (iii) is now immediate froms parts (i) and (ii).
\end{proof}

\begin{theorem}\label{MainResultForSigma01}
    If $R$ is a computable integral domain for which HTP$(K)$ is undecidable, then we have the following:
    \begin{enumerate}[(i)]
        \item For $\delta > 0$, the set $IADR_R(\delta)\cap H_R$ is $\Sigma_1^0$-complete.

        \item For $\delta > 0$, the set $IMDR_R(\delta)\cap T_R$ is $\Sigma_1^0$-complete.

        \item For $\ell \in \mathbb{N}$, the sets $PR_R(\ell), IPR_R(\ell), PR_R(\ell)\cap H_R$ and $IPR_R(\ell)\cap H_R$ are $\Sigma_1^0$-complete.

        \item The sets $PR_R,IPR_R,PR_R\cap H_R,$ and $IPR_R\cap H_R$ are $\Sigma_1^0$-hard.
    \end{enumerate}
    In fact, all of the sets in (i)-(iii) are $\Sigma_1^0$-universal.
\end{theorem}

\begin{proof}
In all 3 cases, we will give a computable reduction either to HTP$(K)$ or to HTP$(K^\times)$. Let $P \in R[x_1,\cdots,x_k]$ be arbitrary.

We now show part (i), that $IADR_R(\delta)\cap T_R$ is $\Sigma_1^0$-complete. Let

\begin{alignat*}{2}
    &P_1(z_1,\cdots,z_{4k}) = P\left(\frac{z_1-z_2}{z_3-z_4},\cdots,\frac{z_{4k-3}-z_{4k-2}}{z_{4k-1}-z_{4k}}\right)\left(\prod_{i = 1}^n(z_{4i-1}-z_{4i})\right)^{\text{deg}(P)}
\end{alignat*}
For the first direction, let us assume that $P(s_1,\cdots,s_k) = 0$ for some $s_1,\cdots,s_k \in K^\times$. Lemma \ref{DensityLemmaWithHomogeneity}(i) tells us that for any $B \subseteq R\setminus\{0\}$ with $d^*(B) > 0$, there exists distinct $z_1,\cdots,z_{4k} \in B$ for which $\frac{z_{4i-3}-z_{4i-2}}{z_{4i-1}-z_{4i}} = s_i$ for all $1 \le i \le k$, so $P_1 \in IADR_R \subseteq IADR_R(\delta)$. Since $P_1 \in T_R$, we see that $P_1 \in IADR_R(\delta)\cap T_R$. Conversely, if $P$ has no root in $K^\times$, then $P_1$ has no root $(z_1,\cdots,z_{4k})$  in $K^\times$ for which $z_i \neq z_j$ when $i \neq j$, so $P_1 \notin IADR_R(\delta)\cap T_R$. We have thus given a computable reduction from HTP$(K^\times)$ to $IADR_R(\delta)\cap T_R$.

We now show part (ii), that $IMDR_R(\delta)\cap H_R$ is $\Sigma_1^0$-complete. For the first direction, let us assume that $P(s_1,\cdots,s_k) = 0$ for some $s_1,\cdots,s_k \in K^\times$. Lemma \ref{DensityLemmaWithHomogeneity}(ii) tells us that for any $B \subseteq R\setminus\{0\}$ with $d^*_\times(B) > 0$, there exists distinct $z_1,\cdots,z_{4k} \in B$ for which $\frac{z_{4i-3}-z_{4i-2}}{z_{4i-1}-z_{4i}} = s_i$ for all $1 \le i \le k$, so $P_1 \in IMDR_R \subseteq IMDR_R(\delta)$. Since $P_1 \in H_R$, we see that $P_1 \in IMDR_R(\delta)\cap H_R$. Conversely, if $P$ has no roots in $K^\times$, then $P_1$ has no root $(z_1,\cdots,z_{4k})$ in $K$ for which $z_i \neq z_j$ when $i \neq j$, so $P_1 \notin IMDR_R(\delta)\cap H_R$.  We have thus given a computable reduction from HTP$(K^\times)$ to $IMDR_R(\delta)\cap H_R$.

We will now show parts (iii) and (iv). We form a new polynomial 

\begin{equation}
    P_2\left(z_1,\cdots,z_{3k}\right) = P\left(\frac{z_1-z_2}{z_3},\cdots,\frac{z_{3k-2}-z_{3k-1}}{z_{3k}}\right)\left(\prod_{i = 1}^kz_{3i}\right)^{\text{deg}(P)}.
\end{equation} 
For the first direction, let us assume that $P(s_1,\cdots,s_k) = 0$ for some $s_1,\cdots,s_k \in K$. Corollary \ref{PartitionLemmaWithHomogeneity} tells us that for any partition of the form $R\setminus\{0\} = \bigcup_{i = 1}^\ell C_i$, there exists $1 \le i_0 \le \ell$ and $z_1,\cdots,z_{3k} \in C_{i_0}$ for which $\frac{z_{3i-2}-z_{3i-1}}{z_{3i}} = s_i$ for all $1 \le i \le k$. 
Since $P_2 \in H_R$, we see that $P_2 \in PR_R\cap H_R \subseteq PR_R(\ell)\cap H_R$. 
Conversely, if $P$ has no roots in $K$, then $P_2$ has no roots in $K$, so $P_2 \notin PR_R(\ell)$.
We have thus given a simultaneous computable reduction from HTP$(K)$ to $PR_R\cap H_R, PR_R, PR_R(\ell)\cap H_R$ and $PR_R(\ell)$. 

To give a simultaneous computable reduction from HTP$(K^\times)$ to $IPR_R\cap H_R, IPR_R, IPR_R(\ell)\cap H_R$ and $IPR_R(\ell)$, we proceed as above and observe the following.
In the proof of the first direction, Corollary \ref{PartitionLemmaWithHomogeneity} allows us to take $z_1,\cdots,z_{3k}$ to be distinct since $s_1,\cdots,s_k \in K^\times$. 
Furthermore, in the proof of the converse direction, we assume that $P$ has no roots in $K^\times$, so $P_2$ has no injective roots in $K$.
\end{proof}

\begin{theorem}\label{MainResultForPi02}
    Let $R$ be a computable integral domain that admits master polynomials.
    \begin{enumerate}[(i)]            
        \item  The set $IMDR_R\cap H_R$ is $\Pi_2^0$-complete.

        \item The sets $PR_R\cap H_R, PR_R,IPR_R\cap H_R,$ and $IPR_R$ are $\Pi_2^0$-complete.
    \end{enumerate}
    In fact, all of the above sets are $\Pi_2^0$-universal.
\end{theorem}

\begin{proof}
Let $A \subseteq \mathbb{N}$ be an arbitrary $\Pi_2^0$ set, and let $S \subseteq \mathbb{N}\times\mathbb{N}$ be a $\Sigma_1^0$-set for which we have $A(m)$ if and only if $S(m,n)$ for all $n \in \mathbb{N}$. Let $S' \subseteq \mathbb{N}\times\mathbb{Z}$ and $S'' \subseteq \mathbb{N}\times\mathbb{Q}$ be as in Section \ref{Hilberts10thProblemSubsection}. Let $(MP(x,y,z_1,\cdots,z_k),f,g)$ be a master polynomial for $S'$ or $S''$ depending on the situation (recall Definition \ref{def:mp}).

We now show part (i), that $IMDR_R\cap H_R$ is $\Pi_2^0$-complete.
For each $m \in \mathbb{N}$, let

\begin{alignat*}{2}
    &MP_1(f(m),y_1,y_2,z_1,\cdots,z_{4k}) =\\
    &MP\left(f(m),\frac{y_1}{y_2},\frac{z_1-z_2}{z_3-z_4},\cdots,\frac{z_{4k-3}-z_{4k-2}}{z_{4k-1}-z_{4k}}\right)\left(y_2\prod_{i = 1}^n(z_{4i-1}-z_{4i})\right)^{\text{deg}(MP)}
\end{alignat*}
We observe that for all $m \in \mathbb{N}$, we have $MP_1\left(f(m),y_1,y_2,z_1,\cdots,z_{4k}\right) \in H_R$. For the first direction, let us assume $A(m)$, so $S'(m,n)$ for all $n \in \mathbb{Z}\setminus\{0\}$ (or $S''(m,n)$ for all $n \in \mathbb{Q}\setminus\{0\})$. 
Let $B \subseteq R\setminus\{0\}$ be such that $d^*_\times(B) > 0$ and let $y_2 \in B$ be arbitrary.
Since $d^*_\times(g^{-1}(\{0\})) = 0$, we see that $d^*_\times(B\setminus y_2g^{-1}(\{0\})) = d^*_\times(B) > 0$, so let $y_1 \in B\setminus y_2g^{-1}(\{0\})$ be arbitrary. 
Since $g\left(\frac{y_1}{y_2}\right) \neq 0$, we may use Lemma \ref{DensityLemmaWithHomogeneity}(ii) to see that $MP_1(f(m),y_1,y_2,z_1,\cdots,z_{4k}) = 0$ has a solution with $z_1,\cdots,z_{4k} \in B\setminus\{y_1,y_2\}$ all distinct, so $MP_1\left(f(m),y_1,y_2,z_1,\cdots,z_{4k}\right) \in IMDR_R$. 

Conversely, let us assume that $A(m)$ does not hold.

We first consider the case in which we have an additive master polynomial, i.e., $g$ is a homomorphism from $(K^\times,\cdot)$ to $(\mathbb{Z},+)$.
There exists $n_m$ such that for all $v \in g^{-1}(\{|n| > n_m\}\cup\{0\})$, $MP(f(m),v,z_1,\cdots,z_k) = 0$ has no solutions in $K$.
Let $r_0 \in K^\times$ be such that $g(r_0) = 1$, let $B = g^{-1}(\{k(n_m+1)\}_{k \in \mathbb{Z}})$, and let $\mathcal{F}$ be a F\o lner sequence in $(R\setminus\{0\},\cdot)$.
Observe that $K^\times = \bigcup_{k = -n_m}^0r_0^kB$, hence $d^*_\times(B\cap R) \ge d_\mathcal{F}(B\cap R) = d_{\mathcal{F}}(B) = \frac{1}{n_m+1} > 0$. 
However, for $y_1,y_2 \in B\cap R\setminus\{0\}$, we have $g\left(\frac{y_1}{y_2}\right) \in (n_m+1)\mathbb{Z}$, so $g\left(\frac{y_2}{y_1}\right) \notin [-n_m,n_m]\setminus\{0\}$.
It follows that for any $y_1,y_2 \in B$, $MP_2(m,y_1,y_2,z_1,\cdots,z_{4k})$ has no injective root $z_1,\cdots,z_{4k} \in K$, hence $MP_1(f(m),y_1,y_2,z_1,\cdots,z_{4k}) \notin IMDR_R$.

Now we consider the case in which we have a multiplicative master polynomial, i.e., $g$ is a homomorphism from $(K^\times,\cdot)$ to $(\mathbb{Q}^\times,\cdot)$.
Let $\phi:\mathbb{Q}\rightarrow\mathbb{Z}$ and $\le_\phi$ be as in Section \ref{Hilberts10thProblemSubsection}, and let $n_m$ such that for all $v \in g^{-1}(\{|n| \ge_\phi n_m+1\}\cup\{0\})$, $MP(f(m),v,z_1,\cdots,z_k) = 0$ has no solutions in $K$.
Let $T$ be a monotile for $(\mathbb{Q}\setminus\{0\},\cdot)$ for which $\{n \in \mathbb{Q}\ |\ -n_m \le_\phi n \le_\phi n_m\}\setminus\{0\} \subseteq T$, and let $C$ be a center set for $T$.
Let $B = g^{-1}(C)$, and observe that $K^\times = \bigcup_{t \in T}r_tB$, where $r_t \in g^{-1}(\{t\})$ is arbitrary, hence $d^*_\times(B\cap R) = \frac{1}{|T|} > 0$.
However, for $y_1,y_2 \in B\setminus\{0\}$, we have $g\left(\frac{y_1}{y_2}\right) \in CC^{-1}$, hence $g\left(\frac{y_1}{y_2}\right) \notin T$.
It follows that for any $y_1,y_2 \in B$, $MP_2(m,y_1,y_2,z_1,\cdots,z_{4k})$ has no injective root $z_1,\cdots,z_{4k} \in K$, hence $MP_1(f(m),y_1,y_2,z_1,\cdots,z_{4k}) \notin IMDR_R$.

We have thus given in both cases a computable reduction from $A$ to $IMDR_R\cap H_R$, and in fact we have simultaneously given a computable reduction from $A$ to $IMDR_R$.

We will now show part (ii). 
For each $m \in \mathbb{N}$ we form a new polynomial 

\begin{equation}
    MP_2\left(x,y_1,y_2,z_1,\cdots,z_{3k}\right) = MP\left(x,\frac{y_1}{y_2},\frac{z_1-z_2}{z_3},\cdots,\frac{z_{3k-2}-z_{3k-1}}{z_{3k}}\right)\left(y_2\prod_{i = 1}^kz_{3i}\right)^{\text{deg}(MP)}.
\end{equation} 
We observe that for each $m \in \mathbb{N}$, we have $MP_2(f(m),y_1,y_2,z_1,\cdots,z_{4k}) \in H_R$. 
Let us assume $A(m)$, so $S'(m,n)$ for all $n \in \mathbb{Z}\setminus\{0\}$ (or $S''(m,n)$ for all $n \in \mathbb{Q}\setminus\{0\})$. In particular, for all $n \in \mathbb{Z}\setminus\{0\}$ (all $n \in \mathbb{Q}\setminus\{0\})$ and for all  $b \in R$ with $g(b) = n$, there exists $s_1(b),\cdots,s_k(b) \in K^\times$ for which $MP(f(m),b,s_1(b),\cdots,s_k(b)) = 0$. 
Since $d^*_\times(g^{-1}(\{0\})) = 0$, we see that $d^*_\times(g^{-1}(\{0\})\cap R) = 0$, so $g^{-1}(\{0\})\cap R$ is not multiplicatively piecewise syndetic.
Lemma \ref{EnhancedPartitionLemmaWithHomogeneity} tells us that for any partition $R\setminus\{0\} = \bigcup_{i = 1}^\ell C_i$, there exists a cell $C_{i_0}$ containing distinct $y_1,y_2,z_1,\cdots,z_{3n}$ such that $b := \frac{y_1}{y_2} \in R\setminus g^{-1}(\{0\})$ and $\frac{z_{3i-2}-z_{3i-1}}{z_{3i}} = s_i(b)$ for all $1 \le i \le k$, so $MP_2(f(m),y_1,y_2,z_1,\cdots,z_{4k}) \in IPR_R$.

Conversely, let us assume that $A(m)$ does not hold.

We first consider the case in which we have an additive master polynomial, i.e., $g$ is a homomorphism from $(K^\times,\cdot)$ to $(\mathbb{Z},+)$.
There exists $n_m$ such that for all $b \in g^{-1}(\{|n| > n_m\}\cup\{0\})$ there is no solution to $MP(f(m),b,z_1,\cdots,z_k) = 0$ in $K$. 
Let $K^\times = \bigcup_{i = 0}^{n_m}C_i$ be the partition given by $C_i = g^{-1}((n_m+1)\mathbb{Z}+i)$. 
We see that if $y_1,y_2 \in C_i$ for some $i$, then $n := g\left(\frac{y_1}{y_2}\right) \in (n_m+1)\mathbb{Z}$. 
It follows that $MP(f(m),\frac{y_1}{y_2},z_1,\cdots,z_k) = 0$ has no solutions, hence $MP_2(f(m),y_1,y_2,z_1,\cdots,z_{3k}) = 0$ is not partition regular over $K^\times$, i.e., $MP_2(f(m),y_1,y_2,z_1,\cdots,z_{3k}) \notin PR_R$.

Now we consider the case in which we have a multiplicative master polynomial, i.e., $g$ is a homomorphism from $(K^\times,\cdot)$ to $(\mathbb{Q}^\times,\cdot)$.
Let $\phi:\mathbb{Q}\rightarrow\mathbb{Z}$ and $\le_\phi$ be as in Section \ref{Hilberts10thProblemSubsection}, and let $n_m$ such that for all $v \in g^{-1}(\{|n| \ge_\phi n_m+1\}\cup\{0\})$, $MP(f(m),v,z_1,\cdots,z_k) = 0$ has no solutions in $K$.
Let $T$ be a monotile for $(\mathbb{Q}^\times,\cdot)$ for which $\{n \in \mathbb{Q}\ |\ -n_m \le_\phi n \le_\phi n_m\}\setminus\{0\} \subseteq T$, and let $C$ be a center set for $T$.
Consider the partition $K^\times = \bigcup_{t \in T}g^{-1}(Ct)$, and observe for $y_1,y_2 \in Ct$, we have $g\left(\frac{y_1}{y_2}\right) \in CC^{-1}$, hence $g\left(\frac{y_1}{y_2}\right) \notin T$.
It follows that for any $y_1,y_2 \in B$, $MP_2(m,y_1,y_2,z_1,\cdots,z_{4k})$ has no root $z_1,\cdots,z_{4k} \in K$, hence $MP_1(f(m),y_1,y_2,z_1,\cdots,z_{4k}) \notin PR_R$.

We have thus given in both cases simultaneous computable reductions from $A$ to $PR_R\cap H_R$, $PR_R$, $IPR_R\cap H_R$, and $IPR_R$. 
\end{proof}

\begin{theorem}\label{Pi02ForAdditiveDensityRegularity}\ 
    \begin{enumerate}[(i)]
        \item If HTP$(\mathbb{Q})$ is undecidable, then $IADR_{\mathbb{Z}}\cap T_{\mathbb{Z}}$ is $\Pi_2^0$-complete, hence $IADR_{\mathbb{Z}}$ is $\Pi_2^0$-hard.

        \item If $K$ is an algebraic function field over a finite field of odd characteristic, or $K = \mathbb{F}_{2^\ell}(t)$, and $R$ is the ring of integers of $K$, then $IADR_K$ is $\Pi_2^0$-hard.
    \end{enumerate}    
    In fact, all of the above sets are $\Pi_2^0$-universal.
\end{theorem}

\begin{proof}
    Let $A \subseteq \mathbb{N}$ be an arbitrary $\Pi_2^0$ set, and let $S \subseteq \mathbb{N}\times\mathbb{N}$ be a $\Sigma_1^0$-set for which we have $A(m)$ if and only if $S(m,n)$ for all $n \in \mathbb{N}$. 
    Let $S' \subseteq \mathbb{N}\times\mathbb{Z}$ and $S'' \subseteq \mathbb{N}\times\mathbb{Q}$ be as in Section \ref{Hilberts10thProblemSubsection}.
    
    We first prove $(i)$. 
    Let $(MP,\text{Id},\text{Id})$ be a multiplicative master polynomial for $S''$. 
    For each $m \in \mathbb{N}$ we form a new polynomial 

\begin{alignat*}{2}
    &MP_1\left(x,y_1,y_2,z_1,\cdots,z_{4k}\right) =\\
    &MP\left(x,y_1-y_2,\frac{z_1-z_2}{z_3-z_4},\cdots,\frac{z_{4k-3}-z_{4k-2}}{z_{4k-1}-z_{4k}}\right)\left(\prod_{i = 1}^k(z_{4i-1}-z_{4i})\right)^{\text{deg}(MP)}.
\end{alignat*} 
We observe that for each $m \in \mathbb{N}$, we have $MP_1(m,y_1,y_2,z_1,\cdots,z_{4k}) \in T_{\mathbb{Z}}$. Let us assume $A(m)$, so for all $n \in \mathbb{Q}\setminus\{0\}$, there exists $s_1(n),\cdots,s_k(n) \in K$ for which $MP(m,n,s_1(n),\cdots,s_k(n)) = 0$. Let $B \subseteq \mathbb{Z}$ be such that $d^*(B) > 0$, let $y_1 \neq y_2 \in B$ be arbitrary, and let $n = y_1-y_2$. Lemma \ref{DensityLemmaWithHomogeneity}(i) tells us that $B\setminus\{y_1,y_2\}$ contains distinct $z_1,\cdots,z_{4k}$ for which $\frac{z_{4i-3}-z_{4i-2}}{z_{4i-1}-z_{4i}} = s_i(n)$ for all $1 \le i \le k$, so $MP_1 \in IADR_{\mathbb{Z}}$. Conversely, let us assume $\neg A(m)$. 
Let $\phi:\mathbb{Q}\rightarrow\mathbb{Z}$ and $\le_\phi$ be as in Section \ref{Hilberts10thProblemSubsection}. 
There exists $n_m \in \mathbb{Q}$ such that for all $|n| \ge_\phi n_m+1$ and for $n = 0$, there is no solution to $MP(m,n,z_1,\cdots,z_k) = 0$ in $\mathbb{Q}$. 
Let $T$ be a monotile for $(\mathbb{Q},+)$ for which $\{n \in \mathbb{Q}\ |\ -n_m \le_\phi n \le_\phi n_m\} \subseteq T$ and let $C$ be a center set for $T$.
We see that $d^*(C) = \frac{1}{|T|}$ and if $y_1,y_2 \in C$, then $n := y_1-y_2 \in C-C$, hence $n \notin T$. 
It follows that $MP(m,y_1-y_2,z_1,\cdots,z_k) = 0$ has no solutions in $\mathbb{Q}$, hence $MP_1(m,y_1,y_2,z_1,\cdots,z_{4k}) \notin IADR_{\mathbb{Z}}$. 
We have thus given simultaneous computable reductions from $A$ to $IADR_{\mathbb{Z}}\cap T_{\mathbb{Z}}$ and $IADR_{\mathbb{Z}}$. 

We now proceed to prove $(ii)$. 
Let $(MP,f,g)$ be an additive master polynomial for $S'$.
Let $\mathfrak{p}$ be an irreducible of $R$ for which $g(k) = \text{ord}_{\mathfrak{p}}(k)$ for all $k \in K^\times$. 
Let $N \in \mathbb{N}$ be arbitrary.
Since $\left(g^{-1}([-N,N])+\mathfrak{p}^{-n(2N+1)}\right)\cap \left(g^{-1}([-N,N])+\mathfrak{p}^{-m(2N+1)}\right) = \emptyset$ for distinct $m,n \in \mathbb{N}$, we see that $d^*(g^{-1}([-N,N])) = 0$, and in particular, $d^*(g^{-1}(\{0\})) = 0$.
For each $m \in \mathbb{N}$ we form a new polynomial 

\begin{alignat*}{2}
    &MP_2\left(x,y,z_1,\cdots,z_{4k}\right) = MP\left(x,y,\frac{z_1-z_2}{z_3-z_4},\cdots,\frac{z_{4k-3}-z_{4k-2}}{z_{4k-1}-z_{4k}}\right)\left(\prod_{i = 1}^k(z_{4i-1}-z_{4i})\right)^{\text{deg}(MP)}.
\end{alignat*} 
Let us assume $A(m)$, so for all $y \in K\setminus g^{-1}(\{0\})$, there exists $s_1(y),\cdots,s_k(y) \in K^\times$ for which $MP(f(m),y,s_1(y),\cdots,s_k(y)) = 0$. 
Let $B \subseteq K$ be such that $d^*(B) > 0$ and let $y \in B\setminus g^{-1}(\{0\})$ be arbitrary. 
Lemma \ref{DensityLemmaWithHomogeneity}(i) tells us that $B\setminus\{y\}$ contains distinct $z_1,\cdots,z_{4k}$ for which $\frac{z_{4i-3}-z_{4i-2}}{z_{4i-1}-z_{4i}} = s_i(y)$ for all $1 \le i \le k$, so $MP_2 \in IADR_K$. 
Conversely, let us assume $\neg A(m)$, so there exists $n_m$ such that for all $y \in g^{-1}(\{|n| \ge n_m\}\cup\{0\})$, there is no solution to $MP(f(m),y,z_1,\cdots,z_k) = 0$ in $K$. 
Since $d^*(g^{-1}([-n_m,n_m])) = 0$, we see that for $B := g^{-1}(\{|n| > n_m\})$ we have $d^*(B) = 1$. 
We observe that $MP(f(m),y,z_1,\cdots,z_k) = 0$ has no solutions in $B$, hence $MP_2(f(m),y,z_1,\cdots,z_{4k}) \notin IADR_K$. 
We have thus given a computable reduction from $A$ to $IADR_K$. 
\end{proof}

\begin{remark}
    The reason we only obtained $\Pi_2^0$-hardness in part (ii) is that we do not currently have an upper bound on the complexity of $IADR_R$, we only have an upper bound on the complexity of $IADR_R\cap T_R$.
    In part (ii), we could not work with $MP_1$ instead of $MP_2$, because it seems possible, for example, that for any $B \subseteq R$ with $d^*(B) > 0$, there exists $y_1,y_2 \in B$ for which $y_1-y_2 \in g^{-1}(\{-1,1\})$.
    The reason we only obtained $\Pi_2^0$-hardness of $IADR_K$ and not $IADR_R$, is that we were only able to show that $d^*(g^{-1}([-N,N])) = 0$ in $K$, but not in $R$. 
    While multiplicative upper Banach density in $(K^\times,\cdot)$ and $(R\setminus\{0\},\cdot)$ are closely related, this is not the case for additive upper Banach density in $(K,+)$ and $(R,+)$.
\end{remark}

When discussing the polynomial Szemer\'edi theorem in the introduction, we saw that it could be rephrased in terms of partial density regularity of a polynomial equation. Similarly, it was shown in \cite{frantzikinakis2023partition} that for any partition $\mathbb{N} = \bigcup_{i = 1}^\ell C_i$, there exists a $1 \le i_0 \le \ell$ and $x,y \in C_{i_0}$ for which $x^2+y^2 = z^2$ for some $z \in \mathbb{N}$, so partial partition regularity is also of interest. We will now show that questions relating to partial regularity have the same complexity as their analogous that pertain to regularity.

Define $IADR_R(n,\delta)$ to be the collection of polynomials $p(x_1,\cdots,x_m)$ with coefficients in $R$, such that for any $B \subseteq R$ with $d^*(B) \ge \delta$ there exists $a_1,\cdots,a_n \in R\setminus\{0\}$ and $a_{n+1},\cdots,a_m \in B$ for which $p(a_1,\cdots,a_m) = 0$ and $a_i \neq a_j$ when $i \neq j$. Define $PR_R(n,\ell)$ to be the collection of polynomials $p(x_1,\cdots,x_m)$ with coefficients in $R$, such that for any partition $R\setminus\{0\} = \bigcup_{i = 1}^\ell C_i$, there exists $1 \le i_0 \le \ell$, $a_1,\cdots,a_n \in R\setminus\{0\}$, and $a_{n+1},\cdots,a_m \in C_{i_0}$ for which $p(a_1,\cdots,a_n) = 0$. We define $IPR_R(n,\ell), PR_R(n), IPR_R(n),\allowbreak IADR_R(n), IMDR_R(n,\delta),$ and $IMDR_R(n)$ analogously.

\begin{lemma}
    Let $R$ be a computable integral domain for which the field of fractions $K$ is totally real, and let $n \in \mathbb{N}$ be arbitrary. For each of the following pairs of sets $\{A,B\}$, there is a computable reduction from $A$ to $B$, and a computable reduction from $B$ to $A$, so $A$ and $B$ have the same descriptive complexity:        
    \begin{multicols}{2}
        \begin{enumerate}
            \item $\{PR_R(\ell)\cap H_R,PR_R(n,\ell)\cap H_R\}$

            \item $\{PR_R\cap H_R,PR_R(n)\cap H_R\}$

            \item $\{IPR_R(\ell)\cap H_R,IPR_R(n,\ell)\cap H_R\}$

            \item $\{IPR_R\cap H_R,IPR_R(n)\cap H_R\}$

            \item $\{IADR_R(\delta),IADR_R(n,\delta)\}$

            \item $\{IADR_R,IADR_R(n)\}$

            \item $\{IMDR_R(\delta),IMDR_R(n,\delta)\}$

            \item $\{IMDR_R,IMDR_R(n)\}$

            \item $\{IADR_R(\delta)\cap T_R,\allowbreak IADR_R(n,\delta)\cap T_R\}$

            \item $\{IADR_R\cap T_R,IADR_R(n)\cap T_R\}$

            \item $\{IMDR_R(\delta)\cap H_R,IMDR_R(n,\delta)\cap H_R\}$

            \item $\{IMDR_R\cap H_R,IMDR_R(n)\cap H_R\}$
        \end{enumerate}
    \end{multicols}
\end{lemma}

\begin{proof}
    We will give the proof for $\{PR_R(\ell)\cap H_R,PR_R(n,\ell)\cap H_R\}$ and $\{IADR_R(\delta)\cap T_R,IADR_R(n,\delta)\cap T_R\}$, and we remark the the proofs for the other situations are similar.

    To give a computable reduction from $PR_R(\ell)\cap H_R$ to $PR_R(n,\ell)\cap H_R$, we see that for any polynomial $p(x_1,\cdots,x_m)$, we have $p \in PR_R(\ell)\cap H_R$ if and only if we have $p' \in PR_R(n,\ell)\cap H_R$ where
    \begin{equation}
        p'(y_1,\cdots,y_n,x_1,\cdots,x_m) = (y_1^2+\cdots+y_n^2)p(x_1,\cdots,x_m).
    \end{equation}
    To give a computable reduction from $PR_R(n,\ell)\cap H_R$ to $PR_R(\ell)\cap H_R$, we will show that for any polynomial $p(y_1,\cdots,y_n,x_1,\cdots,x_m)$, we have $p \in PR_R(n,\ell)\cap H_R$ if and only if we have $p' \in PR_R(\ell)\cap H_R$, where

    \begin{equation*}
        p'(z_1,\cdots,z_{4n},x_1,\cdots,x_m) = p\left(\frac{z_1-z_2}{z_3}z_4,\cdots,\frac{z_{4n-3}-z_{4n-2}}{z_{4n-1}}z_{4n},x_1,\cdots,x_m\right).
\left(\prod_{i = 1}^nz_{4i-1}\right)^{\text{deg}(p)}    \end{equation*}
    It is clear that if $p' \in PR_R(\ell)\cap H_R$, then $p \in PR_R(n,\ell)\cap H_R$, so it remains to show the other direction. Since $p \in H_R$, we see that $\mathcal{A} := \{\{x_1,\cdots,x_m\} \subseteq R\ |\ \exists\ \{y_1,\cdots,y_n\} \subseteq R\text{ such that }p(y_1,\cdots\allowbreak,y_n,x_1,\cdots,x_m) = 0\}$ is multiplicatively translation invariant.
    Since $p \in PR_R(n,\ell)$, $\mathcal{A}$ is partition regular, so Lemma \ref{TranslationInvariancePartitionTheorem} tells us that any multiplicatively piecewise syndetic set $B$ contains a member of $\mathcal{A}$. Now let $R\setminus\{0\} = \bigcup_{i = 1}^\ell C_i$ be a partition, and assume without loss of generality that $C_1$ is multiplicatively piecewise syndetic. 
    Let $\{x_1,\cdots,x_m\} \in \mathcal{A}$ be such that $x_i \in C_1$ for all $i$, and let $y_1,\cdots,y_n \in R\setminus\{0\}$ be such that $p(y_1,\cdots,y_n,x_1,\cdots,x_m) = 0$. 
    Let $z_4,z_8,\cdots,z_{4n} \in C_1\setminus\{x_1,\cdots,x_m\}$ be arbitrary, and let $$\mathcal{A}_1 = \left\{\{z_1,z_2,z_3,z_5,\cdots,z_{4n-1}\} \subseteq R\ |\ \allowbreak\frac{z_{4i-3}-z_{4i-2}}{z_{4i-1}} = \frac{y_i}{z_{4i}}\ \forall\ \allowbreak 1 \le i \le n\ \&\ z_i\neq z_j\text{ when }i \neq j\right\}.$$ 
    We see that $\mathcal{A}_1$ is multiplicatively translation invariant, and Corollary \ref{PartitionLemmaWithHomogeneity} tells us that $\mathcal{A}_1$ is partition regular, so Lemma \ref{TranslationInvariancePartitionTheorem} tells us that $C_1\setminus\{x_1,\cdots,x_m,z_4,\cdots,z_{4n}\}$ contains a member of $\mathcal{A}_1$, which completes the proof.

    To give a computable reduction from $IADR_R(\delta)\cap T_R$ to $IADR_R(n,\delta)\cap T_R$, we see that for any polynomial $p(x_1,\cdots,x_m)$, we have $p \in IADR_R(\delta)\cap T_R$ if and only if we have $p' \in IADR_R(n,\delta)\cap T_R$ where
    \begin{equation}
        p'(y_1,\cdots,y_n,x_1,\cdots,x_m) = ((y_1-x_1)^2+\cdots+(y_n-x_1)^2)p(x_1,\cdots,x_m).
    \end{equation}
    To give a computable reduction from $IADR_R(n,\delta)\cap T_R$ to $IADR_R(\delta)\cap T_R$, we will show that for any polynomial $p(y_1,\cdots,y_n,x_1,\cdots,x_m)$, we have $p \in IADR_R(n,\delta)\cap T_R$ if and only if we have $p' \in IADR_R(\delta)\cap T_R$, where

    \begin{alignat*}{2}
        &p'(z_1,\cdots,z_n,x_1,\cdots,x_m) = \\
        &p\left(\frac{z_1-z_2}{z_3-z_4}+z_5,\cdots,\frac{z_{5n-4}-z_{5n-3}}{z_{5n-2}-z_{5n-1}}+z_{5n},x_1,\cdots,x_m\right)\left(\prod_{i = 1}^n(z_{5i-2}-z_{5i-1})\right)^{\text{deg}(p)}.
    \end{alignat*}
    Since $p \in T_R$, we see that $p' \in T_R$. We also see that if $p' \in IADR_R(\delta)$, then $p \in IADR_R(n,\delta)$, so it remains to show the other direction. 
    We see that $\mathcal{A} := \{\{x_1,\cdots,x_m\} \subseteq R\ |\ \exists\ \{y_1,\cdots,y_n\} \subseteq R\setminus\{0\}\text{ such that }p(y_1,\cdots,y_n,x_1,\cdots,x_m) = 0\}$ is additively translation invariant. 
    Now let $B \subseteq R$ be such that $d^*(B) \ge \delta$, and let $x_1,\cdots,x_m \in B$ be such that there exist $y_1,\cdots,y_n \in R\setminus\{0\}$ for which $p(y_1,\cdots,y_n,x_1,\cdots,x_m) = 0$. Let $z_5,z_{10},\cdots,z_{5n} \in B\setminus\{x_1,\cdots,x_m\}$ be distinct.
    Lemma \ref{DensityLemmaWithHomogeneity}(i) provides us with distinct $z_1,z_2,z_3,z_4,z_6,\cdots,z_{5n-1} \in B\setminus\{x_1,\cdots,x_m,z_5,\cdots,z_{5n}\}$ for which $\frac{z_{5i-4}-z_{5i-3}}{z_{5i-2}-z_{5i-1}} = y_i-z_{5i}$ for all $1 \le i \le n$, which completes the proof.
\end{proof}
\section{Questions and discussion}
Let us first recall a restatement of Rado's Theorem that appears in \cite[Page 74]{GRSRamseyTheory}.

\begin{theorem}[Rado's Theorem restated]\label{RadosTheoremRestated}
A finite system of homogeneous linear equations $F$ with coefficients in $\mathbb{Z}$ is partition regular over $\mathbb{N}$, if and only if for each prime $p$, the system $F$ has a solution in some cell of the partition $C_p = \{C_{p,i}\}_{i = 1}^{p-1}$, where $C_{p,i}$ is the set of natural numbers whose first nonzero digit in the base $p$ expansion is $i$.
\end{theorem}

Usually Rado's Theorem is stated in terms of the columns condition (see Definition \ref{DefinitionOfColumnsCondition}) instead of the formulation given in Theorem \ref{RadosTheoremRestated}, because the columns condition is a computable condition. 
A significant aspect of Rado's characterization of which finite systems of linear equations are partition regular is precisely the extraction of the columns condition from Theorem \ref{RadosTheoremRestated}.
An attempt to do something similar in the setting of polynomial equations was already done in \cite{NonlinearRado}, which built upon the work in \cite{NonStandardRamseyTheory}. 
In particular, the authors considered a countable collection of partitions of $\mathbb{N}$ that are related to the partitions $C_p$, and deduced that if a polynomial (equation) with coefficients in $\mathbb{Z}$ is partition regular over $\mathbb{N}$ then it satisfies the \textbf{maximal Rado condition}. 
Similarly, they
showed that in some instances the polynomial must also satisfy the \textbf{minimal Rado condition}. 
While we do not explicitly define the maximal and minimal Rado conditions here, one can check that both Rado conditions are $\Pi_1^0$ conditions. 
It is also worth noting that for $n \ge 3$, the polynomial $p(x,y,z) = x^n+y^n-z^n$ satisfies the maximal and minimal Rado conditions, but the equation $p(x,y,z) = 0$ is not partition regular over $\mathbb{Z}\setminus\{0\}$ as a consequence of Fermat's last theorem.
Now let us consider the set $L$ of polynomials $p(x_1,\cdots,x_n) \in \mathbb{Z}[x_1,\cdots,x_n]$ that have infinitely many roots in $\mathbb{N}$ and for which $p(w,\cdots,w)$ is a nonzero homogeneous linear polynomial. Conjecture 2.20 of \cite{NonlinearRado} asserts that for $p \in L$, we have $p \in L\cap PR_{\mathbb{Z}}$ if and only if $p$ satisfies the maximal and minimal Rado conditions. 
Since $PR_{\mathbb{Z}}\cap L$ may have a lower lightface complexity than that of $PR_{\mathbb{Z}}$, we are unable to resolve this conjecture. 
However, the previous discussion motivates the following conjecture.

\begin{conjecture}\label{ConjectureForPartitionRegularity}
    Let $R$ be a computable integral domain. There exists a computable collection of finite partitions $\{\mathcal{C}_n\}_{n = 1}^\infty$ of $R\setminus\{0\}$\footnote{This means that there is an algorithm that takes as input $n \in \mathbb{N}$ and $r \in R\setminus\{0\}$ and determines which cell of $\mathcal{C}_n = \{C_{n,i}\}_{i = 1}^{n_\ell}$ the element $r$ belongs to.} such that $p \in (I)PR_R$ if and only if for every $n \in \mathbb{N}$, $p$ has a(n) (injective) root in some cell of $\mathcal{C}_n$.
\end{conjecture}

Since there are only countably many polynomials with coefficients in $R$, Conjecture \ref{ConjectureForPartitionRegularity} is easily seen to be true if we do not require the collection to be computable. 
If $R$ is an integral domain for which $PR_R$ and $IPR_R$ are $\Pi_2^0$-complete, then we cannot have a computable condition, or even a $\Pi_1^0$ condition characterizing (injective) partition regularity. 
In particular, if Conjecture \ref{ConjectureForPartitionRegularity} is true in one of these cases, then it cannot be simplified.
For another example, consider the following statement which describes a set of polynomials that is $\Sigma_1^0$, hence the statement is false when $PR_R$ is $\Pi_2^0$-complete.\\

\noindent\textbf{False Statement:} For each polynomial $p \in \bigcup_{n = 1}^\infty R[x_1,\cdots,x_n]$ there exists a polynomial $q \in \bigcup_{n = 1}^\infty R[x_1,\cdots,x_n]$ that is a computable function of $p$, such that $p(x_1,\cdots,x_n) = 0$ is partition regular over $R\setminus\{0\}$ if and only if $q$ has a root in $K$.\\

On the other hand, the fact that $PR_R(\ell)$ is $\Sigma_1^0$ tells us that we have the following result when HTP$(K)$ is undecidable, and that it cannot be simpliefied when $PR_R(\ell)$ is $\Sigma_1^0$-complete.

\begin{theorem}\label{LogicalTheoremForellPartitionRegularity}
    For each $p \in R[x_1,\cdots,x_n]$ and each $\ell \in \mathbb{N}$, there exists a $q \in R[x_1,\cdots,x_m]$ that is a computable function of $p$ and $\ell$, such that the equation $p(x_1,\cdots,x_n) = 0$ is $\ell$-partition regular over $R\setminus\{0\}$ if and only if $q$ has a root in $K$.
\end{theorem}

All of the ideas of the previous discussion can be applied to the setting of density regularity as well, which leads us to the following assertions.

\begin{conjecture}\label{ConjectureForDensityRegularity}
    Let $R$ be a computable integral domain. There exists a computable collection $\{B_n\}_{n = 1}^\infty$ of subsets of $R\setminus\{0\}$\footnote{This means that there is an algorithm that takes as input $n \in \mathbb{N}$ and $r \in R\setminus\{0\}$ and determines if $r \in B_n$.} with $d^*(B_n) > 0\ (d^*_\times(B_n) > 0)$, such that $p \in IADR_R\ (p \in IMDR_R)$ if and only if for every $n \in \mathbb{N}$, $p$ has an injective root in $B_n$. 
\end{conjecture}

\noindent\textbf{False Statement:} Suppose that $R$ is a computable integral domain for which $IADR_R\cap T_R\allowbreak\ (IMDR_R\cap H_R)$ is $\Pi_2^0$-complete. For each polynomial $p \in \bigcup_{n = 1}^\infty R[x_1,\cdots,x_n]$ there exists a polynomial $q \in \bigcup_{n = 1}^\infty R[x_1,\cdots,x_n]$ that is a computable function of $p$, such that $p \in IADR_R\cap T_R\ (p \in IMDR_R\cap H_R)$ if and only if $q$ has a root in $K$.

\begin{theorem}
    For each $p \in R[x_1,\cdots,x_n]$ and each $\delta \in (0,1]$ there exists a $q \in R[x_1,\cdots,x_m]$ that is a computable function of $p$ and $\delta$, such that $p \in IADR_R(\delta)\cap T_R\ (p \in IMDR_R\cap H_R)$ if and only if $q$ has a root in $K$.
\end{theorem}

In the setting of density regularity, there are still many unresolved questions. The reader will have noticed that despite defining sets such as $IADR_R$ and $IMDR_R(\delta)$ for a computable integral domain $R$, we did not define $ADR_R$ and $MDR_R(\delta)$, so let us do this now. For $\delta > 0$ let $ADR_R(\delta)$ denote the set of polynomials $p \in \bigcup_{n = 1}^\infty R[x_1,\cdots,x_n]$ for which every $B \subseteq R$ with $d^*(B) \ge \delta$ contains a (not necessarily injective) root of $p$. Similarly, for $\delta > 0$ let $MDR_R(\delta)$ denote the set of polynomials $p \in \bigcup_{n = 1}^\infty R[x_1,\cdots,x_n]$ for which every $B \subseteq R$ with $d^*_\times(B) \ge \delta$ contains a (not necessarily injective) root of $p$. Let $ADR_R := \bigcap_{\delta > 0}ADR_R(\delta)$ and let $MDR_R := \bigcap_{\delta > 0}MDR_R(\delta)$. 

\begin{question}\label{QuestionAboutDensityLightFaceComplexity}
    Let $R$ be a computable integral domain. What are the lightface complexities of the sets $ADR_R(\delta), MDR_R(\delta), ADR_R, IADR_R, MDR_R, IMDR_R,\allowbreak ADR_R(\delta)\cap T_R,\allowbreak MDR_R(\delta)\cap H_R, ADR_R\cap T_R,\allowbreak$ and $MDR_R\cap H_R$? 
\end{question}

There are two reasons that we were not able to answer Question \ref{QuestionAboutDensityLightFaceComplexity}. Firstly, the upperbound of the lightface complexity of $IADR_R(\delta)\cap T_R$ (for example) used the translation invariance provided by $T_R$ in order to apply the Compactness Principle for Density Ramsey Theory, i.e., to apply Theorem \ref{DensityCompactnessPrinciple}. Secondly, when we consider the polynomials $MP$ and $MP_1$ either from the proof of Theorem \ref{MainResultForPi02} or from the proof of Theorem \ref{Pi02ForAdditiveDensityRegularity}, we used the injectivity of the roots of $MP_1$ in order to associate them with a root of $MP$.

Another direction in which we can ask questions is in regards to the general relationship between Hilbert's tenth problem and the lightface complexity of the various sets considered in this article. 

\begin{question}\label{QuestionAboutStrengtheningResults}
    Let $R$ be a computable integral domain with field of fractions $K$.
    \begin{enumerate}[(i)]
        \item Is there a converse to Theorem \ref{MainResultForSigma01}? For example, if $IMDR_R(\delta)\cap H_R$ is $\Sigma_1^0$-complete, must HTP$(K)$ be undecidable?

        \item Can Theorem \ref{MainResultForPi02} be improved by weakening the assumption about admitting master polynomials? Similarly, can Theorem \ref{Pi02ForAdditiveDensityRegularity} be improved? For example,
        \begin{enumerate}[(a)]
            \item If $IMDR_R(\delta)\cap H_R$ is $\Sigma_1^0$-complete for all $\delta > 0$, must $IMDR_R\cap H_R$ be $\Pi_2^0$-complete?

            \item If $R$ is the ring of integers of one of the fields $K$ discussed in \cite{denef1978diophantine,shlapentokh1996diophantine,eisentrager2017hilbert}, can we show $IMDR_R\cap H_R$ and $(I)PR_R(\cap H_R)$ to be $\Pi_2^0$-complete?

            \item If $K$ is an algebraic function field over a finite field of odd characteristic, or $K = F_{2^\ell}(t)$, and $R$ is the ring of integers of $K$, is $IADR_R\cap T_R$ $\Pi_2^0$-complete?
        \end{enumerate}
    \end{enumerate}
\end{question}

Lastly, we are left with some natural questions regarding density Ramsey theory. We observe that the characterization of upper Banach density given in Lemma \ref{AlternativeCharacterizationOfUBD} can be used to define upper Banach density in any cancellative (not necessarily amenable) semigroup $S$. Consequently, it is natural to ask which of the results of Section 3 generalize to arbitrary cancellative semigroups.

\begin{question}\label{QuestionAboutDensityRamseyTheoryInCancellativeSemigroups}
    Let $S$ be a countably infinite semigroup that embeds in a group.
    \begin{enumerate}[(i)]
        \item Is there a compactness principle for density Ramsey theory on $S$? In particular, if $\mathcal{A} \subseteq \mathscr{P}_f(S)$ is right translation invariant and $\delta > 0$, are the following equivalent?
        \begin{enumerate}
            \item For every $B \subseteq S$ with $d^*(B) \ge \delta$, there exists $A \in \mathcal{A}$ with $A \subseteq B$.

            \item There exists $H \in \mathscr{F}_f(S)$ such that for every $B \subseteq H$ with $|B| \ge \delta$, there exists $A \in \mathcal{A}$ with $A \subseteq B$.
        \end{enumerate}



        \item Is there a uniformity principle for density Ramsey theory on $S$? In particular, if $\mathcal{A} \subseteq \mathscr{P}_f(S)$ is right translation invariant and has the property that there exists $\delta > 0$, such that for all $B \subseteq S$ with $d^*(B) \ge \delta$ there exists $A \in \mathcal{A}$ for which $d^*(\bigcap_{a \in A}a^{-1}B) > 0$, does there exist a finite subcollection $\mathcal{A}' \subseteq \mathcal{A}$ and some $\gamma > 0$ for which at least one of the following holds?
        \begin{enumerate}[(a)]
            \item If $B \subseteq S$ satisfies $d^*(B) \ge \delta$, then there exists $A \in \mathcal{A}'$ for which $d^*(\bigcap_{a \in A}a^{-1}B) \ge \gamma$.
            
            \item For all $n \in \mathbb{N}$ there exists $F_n \in \mathscr{P}_f(S)$ with $|F_n| \ge n$, such that for any $B \subseteq F$ with $|B| \ge \delta|F_n|$, there exists $A \in \mathcal{A}'$ for which $|\bigcap_{a \in A}a^{-1}B| \ge \gamma|F_n|$.

            \item If $(X,\mathscr{B},\mu,(T_s)_{s \in S})$ is a $S$-system, and $B \in \mathscr{B}$ satisfies $\mu(B) \ge \delta$, then there exists $A \in \mathcal{A}'$ for which $\mu\left(\bigcap_{a \in A}T_a^{-1}B\right) \ge \gamma$.
        \end{enumerate}
    \end{enumerate}
\end{question}
\bibliographystyle{abbrv}
\begin{center}
	\bibliography{references}
\end{center}
\end{document}